\documentclass[a4paper]{article}

\usepackage{amsmath,amssymb, amsbsy}
\usepackage{amsthm}
\usepackage{tikz}
\usepackage[all]{xy}
 \usepackage[colorlinks=true]{hyperref}
\hypersetup{urlcolor=blue, citecolor=red}
\usepackage{hyperref}
\usepackage{verbatim}
\usepackage{appendix}

\newtheorem{teor}{Theorem}[section]
\newtheorem{lemma}[teor]{Lemma}

\newtheorem{theorem}[teor]{Theorem}
\newtheorem{proposition}[teor]{Proposition}
\newtheorem{con lemma}[teor]{Continuity Lemma}
\newtheorem{corollary}[teor]{Corollary}
\newtheorem*{conj}{Conjecture}

{\theoremstyle{definition}
\newtheorem{definition}{Definition}[section]
\newtheorem*{notation}{Notations}

\newtheorem{remark}[teor]{Remark}

\newcommand{\R}{\mathbb{R}}

\newcommand{\N}{\mathbb{N}}

\def\beq{\begin{equation}}
\def\eeq{\end{equation}}

\def\pa{\partial}
\def\t{\theta}

\def\d{\delta}
\def\g{\gamma}
\def\s{\sigma}

\def\f{\varphi}
\def\l{\lambda}
\def\a{\alpha}

\def\eps{\varepsilon}

\def\D{\Delta}
\def\S{\mathbb{S}}

\title{\sc Monotonicity and $1$-dimensional symmetry for solutions of an elliptic system arising in Bose-Einstein condensation}
\author{Alberto Farina and Nicola Soave}

\begin{document}



\maketitle
{\footnotesize
\centerline{Alberto Farina}
 \centerline{LAMFA, CNRS UMR 7352, Universit\'e de Picardie Jules Verne}
   \centerline{33 rue Saint-Leu, 80039 Amiens, France}
   \centerline{and}
   \centerline{Institut Camille Jordan, CNRS UMR 5208, Universit\'e Claude Bernard Lyon I}
   \centerline{43 boulevard du 11 novembre 1918, 69622 Villeurbane cedex, France}
   \centerline{email: alberto.farina@u-picardie.fr}

\medskip 
\centerline{Nicola Soave}
 \centerline{Universit\`a degli Studi di Milano - Bicocca, Dipartimento di Ma\-t\-ema\-ti\-ca e Applicazioni}
   \centerline{Via Roberto Cozzi 53, 20125 Milano, Italy}
\centerline{email: n.soave@campus.unimib.it}   }

\begin{abstract}
\noindent We study monotonicity and $1$-dimensional symmetry for positive solutions with algebraic growth of the following elliptic system:
\[
\begin{cases}
-\Delta u = -u v^2 & \text{in $\R^N$}\\
-\Delta v= -u^2 v & \text{in $\R^N$},
\end{cases}
\]
for every dimension $N \ge 2$. In particular, we prove a Gibbons-type conjecture proposed by H. Berestycki, T. C. Lin, J. Wei and C. Zhao.
\end{abstract}


\section{Introduction}

This paper concerns monotonicity and $1$-dimensional symmetry for entire solutions with algebraic growth of the following semilinear elliptic system:
\begin{equation}\label{system}
\begin{cases}
-\Delta u= - u v^2 & \text{in $\R^N$} \\
-\Delta v=- u^2 v & \text{in $\R^N$} \\
u, v>0 & \text{in $\R^N$},
\end{cases}
\end{equation}
where $N \ge 2$. System \eqref{system} has been intensively studied during the last years, starting from the seminal papers \cite{BeLiWeZh} and \cite{NoTaTeVe}. Therein, \eqref{system} appears in the analysis of phase-separation phenomena for Bose-Einstein condensates with multiple states (we refer to \cite{BeLiWeZh, BeTeWaWe} and to the references therein for more details concerning the physical motivations). In particular, in \cite{BeLiWeZh} is emphasized the relationship between system \eqref{system} and the celebrated Allen-Cahn equation. This relationship induced the authors to formulate a De Giorgi's-type and a Gibbons'-type conjecture for the solutions of \eqref{system} (we refer to \cite{FaVa} for a review on the De Giorgi's conjecture and some related problems). In this paper we address precisely the following Gibbons'-type conjecture:
\begin{conj}[section 7 of \cite{BeLiWeZh}]
Let $N \ge 2$, let $(u,v)$ be a solution of \eqref{system} satisfying 
\[
\begin{split}
\lim_{x_N \to -\infty} u(x',x_N)= 0 \quad \text{and} \quad \lim_{x_N \to +\infty} u(x',x_N)= +\infty  \\
\lim_{x_N \to -\infty} v(x',x_N)= +\infty \quad \text{and} \quad \lim_{x_N \to +\infty} v(x',x_N)= 0, 
\end{split}
\]
the limits being uniform in $x' \in \R^{N-1}$. Then $(u,v)$ is $1$-dimensional.
\end{conj}
\noindent Clearly, with respect to the original counterparts, major difficulties arise from the fact that in the present case we have to deal with a system of equations instead of with a single equation, and with unbounded solutions. \\
In what follows, we review the main achievements concerning the existence and the $1$-dimensional symmetry of entire solutions to \eqref{system}. In \cite{NoTaTeVe}, it is showed that there is not a positive solution which is globally $\a$-H\"older continuous for some $\a \in (0,1)$. On the other hand, in \cite{BeLiWeZh} the authors proved the existence of a non-constant solution for \eqref{system} when $N=1$ (in this case we have a system of ODEs). This solution has linear growth: there exists $C>0$ such that
\[
u(t)+v(t) \le C(1+|t|) \qquad \forall t \in \R;
\]
moreover, it is reflectionally symmetric with respect to a certain $t_0 \in \R$, in the sense that 
\[
u(t_0+t)= v(t_0-t) \qquad \forall t \in \R.
\]
In \cite{BeTeWaWe} it is proved that this is the unique positive entire solution (up to translations and scalings) in case $N=1$. On the other hand, always in \cite{BeTeWaWe}, the authors constructed for every $N \ge 2$ entire solutions with arbitrary integer algebraic growth; here and in the rest of the paper we say that $(u,v)$ has algebraic growth if there exist $p \ge 1$ and $C>0$ such that
\begin{equation}\tag{h1}\label{alg growth}
u(x)+v(x) \le C(1 +|x|^p) \qquad \forall x \in \R^N.
\end{equation}
These solutions, which depend on more then one variable, are constructed exploiting the deep relationship between entire solutions of \eqref{system} and entire harmonic functions. This relationship has been established in \cite{DaWaZh, NoTaTeVe, TaTe}. Recently, a similar argument has been exploited in \cite{SoZi} to prove the existence of solutions to \eqref{system} having exponential growth in one direction.\\
Concerning symmetry results, we say that $(u,v)$ is $1$-dimensional if there exists $\nu \in \R^N$ such that
\[
u(x)=\bar u(\langle \nu,x \rangle) \quad \text{and} \quad v(x) = \bar v(\langle \nu, x \rangle),
\]
for some $\bar u,\bar v: \R \to \R$. In \cite{BeLiWeZh} the authors proved that if $N=2$, $(u,v)$ has linear growth and is monotone in the $e_N$ direction, in the sense that
\[
\frac{\pa u}{\pa x_N}>0 \quad \text{and} \quad \frac{\pa v}{\pa x_N}<0 \quad \text{in $\R^N$},
\]
then $(u,v)$ is $1$-dimensional. An improvement of this result has been recently obtained by the first author in \cite{Fa2}: he replaced the linear growth condition with an arbitrary algebraic growth condition (i.e. \eqref{alg growth}), and weakened the monotonicity assumption requiring that only one component between $u$ and $v$ is monotone in $x_N$. Always in case $N=2$, in \cite{BeTeWaWe} it is showed that if $(u,v)$ has linear growth and is stable then $(u,v)$ is $1$-dimensional. As far as the case $N \ge 2$ is concerned, we refer to the recent contribution \cite{Wa}: the author proved that for any $N \ge 2$, if $(u,v)$ has linear growth and is a local minimizer for the energy functional, then $(u,v)$ is $1$-dimensional.\\
Our main result is the following:
\begin{theorem}\label{main thm}
Let $N \ge 2$, let $(u,v)$ be a solution of system \eqref{system} having algebraic growth (i.e. satisfying \eqref{alg growth})  and such that
\begin{equation}\tag{h2}\label{uniform limit}
\begin{split}
\lim_{x_N \to -\infty} u(x',x_N)= 0 \quad \text{and} \quad \lim_{x_N \to +\infty} u(x',x_N)= +\infty  \\
\lim_{x_N \to -\infty} v(x',x_N)= +\infty \quad \text{and} \quad \lim_{x_N \to +\infty} v(x',x_N)= 0, 
\end{split}
\end{equation}
the limit being uniform in $x' \in \R^{N-1}$. Then $(u,v)$ depends only on the $x_N$ variable, and
\[
\frac{\pa u}{\pa x_N}  >0 \quad \text{and} \quad \frac{\pa v}{\pa x_N} <0 \quad \text{in $\R^N$}.
\]
\end{theorem}
Some remarks are in order: the conjecture proposed by H. Berestycki, T. C. Lin, J. Wei and C. Zhao in \cite{BeLiWeZh} was formulated without assumption \eqref{alg growth}. Nevertheless, at this stage it seems really hard to deal without an algebraic growth condition, because most of the results which are present in the literature rest strongly on it (concerning symmetry results, except the work \cite{Fa2} all the quoted achievements are obtained under the linear growth assumption). As far as we know, the unique contribution going beyond the algebraic growth is given in \cite{SoZi}, where the authors proved the existence of solutions to \eqref{system} with exponential growth. Therein, it is often remarked the striking difference between solutions having algebraic growth and solutions having exponential growth, which reflects the difference between harmonic polynomial and harmonic function with exponential growth. For us, the main problem to deal with solutions not satisfying the algebraic growth condition would be the lack of the blow-down technology, see Theorem 1.4 of \cite{BeTeWaWe}. \\
On the other hand, in light of the strongly coupled nature of system \eqref{system}, we can weaken assumption \eqref{uniform limit} obtaining again monotonicity and $1$-dimensional symmetry.

\begin{corollary}\label{main corol}
Let $N \ge 2$, and let $(u,v)$ be a solution of system \eqref{system} having algebraic growth (i.e. satisfying \eqref{alg growth}), and such that
\begin{equation}\tag{h3}\label{limit at infinity}
\lim_{x_N \to \pm \infty} \left(u(x',x_N) - v(x',x_N) \right)= \pm \infty,
\end{equation}
the limits being uniform in $x' \in \R^{N-1}$. Then $(u,v)$ depends only on the $x_N$ variable, and
\[
\frac{\pa u}{\pa x_N}  >0 \quad \text{and} \quad \frac{\pa v}{\pa x_N} <0 \quad \text{in $\R^N$}.
\]
\end{corollary}

\begin{notation}
Let $(u,v)$ be a solution of \eqref{system}. We recall some notation that are by now standard. Given $x \in \R^N$ and $r>0$, we set
\begin{equation}\label{def di H,N}
\begin{split}
H(x,r)& := \frac{1}{r^{N-1}} \int_{\pa B_r(x)} u^2+v^2, \\
E(x,r) & := \frac{1}{r^{N-2}} \int_{B_r(x)}  |\nabla u|^2 + |\nabla v^2| +u^2 v^2,
\end{split}
\end{equation}
and $ \displaystyle N(x,r) := \frac{E(x,r)}{H(x,r)}$. The function $N$ is called \emph{Almgren frequency function}, or \emph{Almgren quotient}. \\
For every $x_0 \in \R^N$ and $R>0$, we introduce
\begin{equation}\label{def blow-down}
\left(u_{x_0,R}(x), v_{x_0,R}(x) \right) := \left( \frac{1}{\sqrt{H(x_0,R)}} u(x_0+Rx), \frac{1}{\sqrt{H(x_0,R)}} v(x_0 +Rx) \right).
\end{equation}
The family $\{(u_{x_0,R},v_{x_0,R}): R >0\}$ is called \emph{the blow-down family of $(u,v)$ centered in $x_0$}. \\
Finally, we will consider the function
\[
J(x,r) := \frac{1}{r^4} \int_{B_r(x)} \frac{|\nabla u(y)|^2+u^2(y) v^2(y)}{|x-y|^{N-2}}\,dy \int_{B_r(x)} \frac{|\nabla v(y)|^2+u^2(y) v^2(y)}{|x-y|^{N-2}}\,dy.
\]
For some properties related to the Almgren quotient, the blow-down family and the function $J$, we refer to the appendix and to the references therein. 
\medskip

\noindent We will use the notation $x =(x',x_N) \in \R^{N-1} \times \R$ for a point of $\R^N$. \\
The directional derivative with respect to $\mu \in \S^{N-1}$ will be denoted by $\frac{\pa}{\pa \mu}$ or by $\pa_\mu$. When we integrate by parts, we denote by $\pa_\nu$ the normal derivative. The $i$-th coordinate direction will be denoted by $e_i$.\\
We will use the notation $\left\langle \cdot,\cdot\right\rangle $ or $| \cdot |$ for the usual scalar product or the usual euclidean norm in any euclidean space. \\
Throughout the paper $C, C_1, C_2, \ldots$ will denote positive constants which may refer to different quantities from line to line. On the other hand, we will fix the value of some constants. In these cases we will use the over-lined notation $\bar C_1, \bar C_2,\ldots$. 
\end{notation}

\paragraph{Plan of the paper}
We wish to prove the $1$-dimensional symmetry of the solution $(u,v)$ by means of the moving planes method. First of all, in section \ref{sec:refined}, we will provide some estimates which will be useful in the rest of the paper.\\
In section \ref{sec: uniqueness} we will make rigorous the intuitive fact that, under assumption \eqref{uniform limit}, $x_N$ is the privileged variable of the solution $(u,v)$: to be precise, by means of the blow-down technology, we will show that independently on the base point $x_0 \in \R^N$ the entire blow-down family converges to the same function $(\gamma x_N^+,\gamma x_N^-)$, with $\gamma>0$. \\
In section \ref{sec: monot at infinity} we will show that, under our assumptions, $\pa_N u(x) >0$ in $\{x_N \gg 1\}$ and $\pa_N v(x)<0$ in $\{x_N \ll 1\}$. This does not follow directly from the results of section \ref{sec: uniqueness}, because the quantitative information given by the convergence of the blow-down family get worse as $R \to +\infty$ (we refer to section \ref{sec: monot at infinity} for more details). \\
In section \ref{sec: monotonicity} we will use the moving planes method to deduce that $\pa_N u >0 $ and $\pa_N v<0 $ in $\R^N$; firstly, by the fact that $\pa_N u > 0$ for $x_N \gg 1$ we will deduce that in the same region $\pa_N v < 0$; this can be done thanks to a version of the maximum principle in unbounded domains, and allow us to start the moving planes method. We point out that it is not possible to proceed separately on $u$ and on $v$ (that is, it is not possible to show that $\pa_N u >0$ and, in a second time, that $\pa_N v<0$ in $\R^N$); this reflects the strongly coupled nature of system \eqref{system}, and introduce a lot of complications with respect to the case of a single equation.\\
In section \ref{sec:symmetry}, we will complete the proof of Theorem \ref{main thm}, passing from the monotonocity in the $e_N$ direction to the monotonocity in all the directions of the upper hemisphere $\S^{N-1}_+:= \{ \nu \in \S^{N-1}: \langle e_N, \nu\rangle>0 \}$; we will follow the line of reasoning introduced by the first author in \cite{Fa1}, with the obvious complications which come from the fact that we are working with a system and not on a single equation, and that we are dealing with unbounded solutions.\\ 
Finally, in section \ref{sec:corollary} we will give the proof of Corollary \ref{main corol}; to be precise, we will show that under \eqref{alg growth} and \eqref{limit at infinity}, the assumption \eqref{uniform limit} is satisfied, so that Corollary \ref{main corol} follows from our main theorem.\\
We reported some known results in the appendix at the end of the paper; this appendix can be considered as an easy-to-read introduction to the study of system \eqref{system}. 
 
\section{Preliminary results}\label{sec:refined}

In \cite{Wa}, the author introduced an Alt-Caffarelli-Friedman monotonicity formula for solutions of \eqref{system} (we reported it in the appendix). This formula gives a lower bound for some integral quantities related to solutions having linear growth (cf. the results of section 4 of \cite{Wa}). In this section we prove some new results and we refine some estimates of the quoted paper, in order to use them in the next sections.

\medskip

In Corollary 4.5 of \cite{Wa}, the author used the linear growth of the solution $(u,v)$ to obtain a lower bound for the growth of the function
\[
r \mapsto \int_{\pa B_r(0)} u^2+v^2.
\]
We think that it is interesting to note that an equivalent estimate holds true assuming only that $(u,v)$ has algebraic growth. Clearly, this requires some extra-work.

\begin{corollary}\label{stima inferiore di crescita universale}
Let $(u,v)$ be a solution of \eqref{system} satisfying \eqref{alg growth}. There exists $C>0$ such that
\[
\int_{B_r(0)} u^2+v^2 \ge C r^{N+2} \qquad \forall r \ge 1.
\]
\end{corollary}
\begin{proof}
Assume by contradiction that the statement is not true: there exists $\eps_n \to 0$ and $(r_n) \subset [1,+\infty)$ such that
\begin{equation}\label{eq26}
\int_{B_{r_n}(0)} u^2+v^2 \le \eps_n r_n^{N+2}.
\end{equation}
\paragraph{Step 1)} \emph{$\displaystyle \liminf_{n \to \infty} r_n= +\infty$}.\\
If not, up to a subsequence $r_n \to \bar r \ge 1$. By the dominated convergence theorem, we deduce
\[
\int_{B_{\bar r}(0)} u^2+v^2 = 0 \quad  \Rightarrow \quad (u,v) \equiv (0,0),
\]
a contradiction.
\paragraph{Step 2)} \emph{Conclusion of the proof}.\\
To simplify the notation, we denote by $j_u(r)$ the quantity
\[
\frac{1}{r^2} \int_{B_r(0)} \frac{|\nabla u(y)|^2 + u^2(y) v^2(y)}{|y|^{N-2}}\,dy,
\]
and by $j_v(r)$ the same quantity for the component $v$. Now, by Theorem \ref{ACF} there exists $C>0$ such that $J(0,r) \ge C$ for every $r \ge 1$, that is, $j_u(r) j_v(r) \ge C$ for every $r \ge 1$. In particular, this holds true for every $r_n$. Up to a subsequence, we can assume $j_u(r_n) \ge C$ for every $n$. By means of \eqref{4.11 by Wang} (we remark that the constant appearing is independent on $r$) plus our absurd assumption \eqref{eq26}, we obtain
\[
0<C \le j_u(r_n) \le \frac{C}{r_n^{N+2}} \int_{B_{2r_n}(0)} u^2 \le  C \eps_n \to 0 
\]
as $n \to \infty$, a contradiction.
\end{proof}

Under the linear growth assumption of $(u,v)$, that is, there exists $C>0$ such that
\begin{equation}\label{linear growth}
u(x)+v(x) \le C(1 +|x|) \qquad \forall x \in \R^N,
\end{equation}
we obtain a uniform (in both $x \in \R^N$ and $r \ge 1$) lower bound for the values $\{H(x,r)\}$.

\begin{lemma}\label{uniform growth for H near u=v}
Let $(u,v)$ be a solution of \eqref{system} with linear growth. There exists $\bar C_1>0$ such that
\[
H(x,r) \ge \bar C_1 
\]
for every $x \in \R^N$ and $r \ge 1$. 
\end{lemma}

\begin{proof}
By the monotonicity of $H(x,\cdot)$, it is sufficient to show that $H(x,1) \ge C$ with $C$ independent on $x \in \R^N$. By contradiction, assume that there exists $(x_i) \subset \R^N$ such that
\begin{equation}\label{abs1}
\lim_{i \to \infty} \int_{\pa B_1(x_i)} u^2+v^2 = 0.
\end{equation}
By Corollary \ref{stima inferiore di crescita universale}, we know that there exists $C>0$ such that 
\[
 \int_{B_r(0)} u^2+v^2 \ge C r^{N+2} \qquad \forall r \ge 1,
 \]
Let $r \ge 1$; for every $i$ we have
\begin{equation}\label{eq21}
 \int_{B_{r+|x_i|}(x_i)} u^2 +v^2 \ge  \int_{B_r(0)} u^2 +v^2 \ge C r^{N+2}. 
\end{equation}
Note that
\[
 \int_{B_{r+|x_i|}(x_i)} u^2 +v^2  = \int_{B_{r+|x_i|}(x_i) \setminus B_{1}(x_i)} u^2 +v^2 + \int_{B_{1}(x_i)} u^2 +v^2;
 \]
thanks to Lemma \ref{from growth to N} we know that $N(x_i,r) \le 1$ for every $r \ge 1$, for every $i$. Hence, by means of Corollary \ref{doubling}, we deduce
\begin{align*}
\int_{B_{r+|x_i|}(x_i) \setminus B_{1}(x_i)} u^2 +v^2 &= \int_{1}^{r+|x_i|} \left(\int_{\pa B_s(x_i)} u^2+v^2 \right) \,ds \le e \left(\int_{\pa B_1(x_i)} u^2+v^2 \right) \int_{1}^{r+|x_i|} s^{N+1}\,ds \\
& \le e \left(\int_{\pa B_1(x_i)} u^2+v^2 \right) (r+|x_i|)^{N+2}.
\end{align*}
Therefore
\begin{equation}\label{eq22}
 \int_{B_{r+|x_i|}(x_i)} u^2 +v^2 \le e \left(\int_{\pa B_1(x_i)} u^2+v^2 \right) (r+|x_i|)^{N+2} + \int_{B_{1}(x_i)} u^2 +v^2.
\end{equation}
We observe that from the linear growth of $(u,v)$ it follows also
\[
\int_{B_{1}(x_i)} u^2 +v^2 \le C (1+|x_i|)^2,
\]
where $C$ does not depend on $i$. Plugging into the \eqref{eq22} and choosing $r=r_i \ge |x_i|$, $r_i \to +\infty$ as $i \to \infty$ (here $i$ is fixed, so this choice is possible), we deduce
\[
\begin{split}
\int_{B_{r_i+|x_i|}(x_i)} u^2 +v^2 & \le  e \left(\int_{\pa B_1(x_i)} u^2+v^2 \right) (r_i+|x_i|)^{N+2} + C \left(1+|x_i|^2\right) \\
& \le C \left(\int_{\pa B_1(x_i)} u^2+v^2 \right) r_i^{N+2} + C (1+ r_i^2).
\end{split}
\]
A comparison with \eqref{eq21} yields
\[
C r_i^{N+2} \le C \left(\int_{\pa B_1(x_i)} u^2+v^2 \right) r_i^{N+2} + C r_i^2.
\]
Dividing for $r_i^{N+2}$ and passing to the limit as $i \to \infty$, we finally obtain a contradiction:
\[
0<C \le C \lim_{i \to \infty} \int_{\pa B_1(x_i)} u^2+v^2 =0,
\]
where we used our absurd assumption, equation \eqref{abs1}.
\end{proof}

Where $|u-v|$ is not too large, it is natural to expect that this provides a lower bound on the integrals of both $u^2$ and $v^2$. To be precise:

\begin{lemma}\label{Wa corretto}
Let $(u,v)$ be a solution of \eqref{system} having linear growth. For every $C_1 < \sqrt{\bar C_1 | \S^{N-1}|}$     (where $\bar C_1$ has been defined in Lemma \ref{uniform growth for H near u=v}) there exists $\bar C_2>0$ such that 
\[
\int_{\pa B_1(x_0)} u^2 \ge \bar C_2 \quad \text{and} \quad \int_{\pa B_1(x_0)} v^2 \ge \bar C_2
\]
for every $x_0 \in \{|u-v| < C_1\}$.
\end{lemma}

\begin{proof}
Without loss of generality, we can assume by contradiction that, for a sequence $(x_i) \subset \{|u-v|<C_1\}$, we have
\[
\lim_{i \to \infty} \int_{\pa B_1(x_i)} u^2= 0.
\]
We claim that under this assumption
\[
\lim_{i \to \infty} \int_{\pa B_1(x_i)} v^2= 0.
\]
If not, up to a subsequence there exists $\delta>0$ such that \( \lim_{i} \int_{\pa B_1(x_i)} v^2 \ge \delta^2\). We introduce the sequence
\[
\left(u_i(x), v_i(x)\right) = \left(\frac{1}{\sqrt{H(x_i,1)}} u(x_i+x), \frac{1}{\sqrt{H(x_i,1)}} v(x_i+x) \right).
\]
Note that $\int_{\pa B_1(0)} u_i^2+v_i^2=1$ for every $i$. Each $(u_i,v_i)$ solves
\[
\begin{cases}
-\Delta u_i = H(x_i,1) u_i v_i^2 & \text{in $\R^N$} \\
- \Delta v_i = H(x_i,1) u_i^2 v_i & \text{in $\R^N$};
\end{cases}
\]
By Corollary \ref{doubling} (which we can apply, see Remark \ref{rem su scaling per N}), we deduce that 
\begin{equation}\label{eq19}
\int_{\pa B_r(0)} u_i^2+v_i^2 \le e r^{N+1} \qquad \forall r, \ \forall i.
\end{equation}
As $u_i$ and $v_i$ are subharmonic, the \eqref{eq19} gives a uniform bound on the $L^\infty(B_{r/2}(0))$ norm of the family $\{(u_i,v_i)\}$, for every $r \ge 1$. Now, we have to distinguish between
\begin{itemize}
\item[($i$)] the sequence $\{H(x_i,1)\}$ is bounded.  
\item[($ii$)] the sequence $\{H(x_i,1)\}$ is unbounded.  
\end{itemize}
\medskip

In case ($i$), up to a subsequence $H(x_i,1) \to H_\infty$. Also, $\{u_i\}, \{v_i\}, \{\Delta u_i\}, \{\Delta v_i\}$ are uniformly bounded in every compact subset $K$ of $\R^N$. By standard gradient estimates for elliptic equations (see \cite{GiTr}) we deduce that $\{ \nabla u_i\}, \{ \nabla v_i\}$ are uniformly locally bounded in $\R^N$, so that up to a subsequence $(u_i,v_i) \to (u_\infty,v_\infty)$ in $\mathcal{C}^2_{loc}(\R^N)$ (to pass from the uniform convergence to the $\mathcal{C}^2$ convergence, we refer to the regularity theory for elliptic equations, e.g. \cite{GiTr}). From the absurd assumption and our normalization it follows
\begin{equation}\label{eq20}
\int_{\pa B_1(0)} u_\infty^2=0 \quad \text{and} \quad \int_{\pa B_1(0)} v_\infty^2 =1.
\end{equation}
Moreover, $u_\infty$ and $v_\infty$ are subharmonic and nonnegative. This implies $u_\infty \equiv 0$ in $B_1(0)$, which in turns yields (apply the strong maximum principle) $u_\infty \equiv 0$ in $\R^N$. Hence, $v_\infty$ is harmonic  and nonnegative in $\R^N$ (this follows by the $\mathcal{C}^2$ convergence): by the Liouville theorem for harmonic functions, $v_\infty \equiv const$. Now, since $x_i \in \{|u-v| < C_1\}$ with $C_1 < \sqrt{\bar C_1 | \S^{N-1}|}$, and in light of Lemma \ref{uniform growth for H near u=v}, we deduce
\[
v_\infty(0)= \lim_{i \to \infty} \left(\frac{1}{\sqrt{H(x_i,1)}} | v(x_i)-u(x_i)| + u_i(0) \right) \le \lim_{i \to \infty}  \left(\frac{1}{\sqrt{\bar C_1}} C_1 + u_i(0) \right)    < \sqrt{|\S^{N-1}|}.
\]
But since $v_\infty$ is constant and \eqref{eq20} holds true, necessarily $v_\infty(0)^2 |\S^{N-1}|=1$, a contradiction.

\medskip

In case ($ii$), up to a subsequence $H(x_i,1) \to +\infty$ as $i \to \infty$. Due to the fact the $\{(u_i,v_i)\}$ is uniformly bounded in every compact subset of $\R^N$, we are in position to apply Theorems \ref{local Holder estimate} and \ref{segregation thm}: for every $K \subset \subset \R^N$, the sequence $\{(u_i,v_i)\}$ is uniformly bounded in $\mathcal{C}^{0,\alpha}(K)$ for every $\alpha \in (0,1)$, and, up to a subsequence, $(u_i,v_i) \to (u_\infty,v_\infty)$ in $\mathcal{C}^0(K) \cap H^1(K)$, where $u_\infty-v_\infty$ is harmonic, $u_\infty$ and $v_\infty$ are subharmonic and \eqref{eq20} holds true. As in the previous case, by subharmonicity, nonnegativity, and the fact that $\int_{\pa B_1(0)} u_\infty^2=0$ we deduce $u_\infty \equiv 0$ in $B_1(0)$. So, $v_\infty$ is nonnegative and harmonic in $B_1(0)$; moreover,
\[
v_\infty(0)= \lim_{i \to \infty} \left(\frac{1}{\sqrt{H(x_i,1)}} | v(x_i)-u(x_i)| + u_i(0) \right) \le \lim_{i \to \infty}  \left(\frac{1}{\sqrt{H(x_i,1)}} C_1 + u_i(0) \right)    = 0;
\]
this implies $v_\infty \equiv 0$ in $B_1(0)$, and gives a contradiction with \eqref{eq20}.

\medskip

We proved that if $\int_{\pa B_1(x_i)} u^2 \to 0$, then $H(x_i,1) \to 0$ as $i \to \infty$. But this is contradiction with Lemma \ref{uniform growth for H near u=v}.
\end{proof}

\begin{remark}\label{def di bar C_3}
From now on we will denote as $\bar C_3$ a fixed positive constant strictly smaller then \\
$\sqrt{\bar C_1 |\S^{N-1}|}$.
\end{remark}

Let's come back to the Alt-Caffarelli-Friedman monotonicity formula, see Theorem \ref{ACF}. In some cases it is possible to get rid of the dependence of the constant $C(x_0)$ on $x_0$. This is the purpose of the following general result, which holds true for solutions with arbitrary algebraic growth and allows $x_0$ to vary in a set of full measure.

\begin{proposition}\label{corol 4.8}
Let $(u,v)$ be a solution of \eqref{system} satisfying \eqref{alg growth}. Assume that 
\begin{equation}\label{eq29}
\int_{\pa B_1(x_0)} u^2 \ge C_1 \quad \text{and} \quad \int_{\pa B_1(x_0)} v^2 \ge C_1 \quad \forall x_0 \in \{|u-v|< \delta\},
\end{equation}
where $C_1,\delta>0$. Then there exists $C_2>0$ such that
\[
r \mapsto e^{-C_2 r^{-1/2}} J(x_0,r) \quad \text{is nondecreasing in $r$}
\]
for every $r \ge 1$, for every $x_0 \in \{|u-v|< \delta\}$.
\end{proposition}

\begin{proof}[Proof (cf. proof of Theorem 4.3 and the observation before Corollary 4.8  in \cite{Wa})]
For any \\
$x_0 \in \{|u-v|<\delta\}$ and $r \ge 1$, we denote
\[
\left(\bar u_{x_0,r}(x), \bar v_{x_0,r}(x) \right) = \left( u(x_0+rx), v(x_0+rx) \right) \qquad \text{with } x \in \pa B_1(0).
\] 
As in the proof of Lemma 2.5 of \cite{NoTaTeVe}, it results
\begin{equation}\label{eq in ACF}
\frac{d}{dr} \log J(x_0,r) \ge -\frac{4}{r} + \frac{2}{r} \left[\Gamma \left( \Lambda_1(x_0,r) \right) + \Gamma \left( \Lambda_2(x_0,r) \right) \right],
\end{equation}
where $\Gamma(t) = \sqrt{\left(\frac{N-2}{2}\right)^2+t } - \left( \frac{N-2}{2} \right)$,
\begin{align*}
\Lambda_1(x_0,r)&= \frac{r^2 \int_{\pa B_r(x_0)} |\nabla_\theta u|^2+u^2 v^2 }{\int_{\pa B_r(x_0)} u^2} = \frac{ \int_{\pa B_1(0)} |\nabla_\theta \bar u_{x_0,r}|^2+r^2 \bar u_{x_0,r}^2 \bar v_{x_0,r}^2 }{\int_{\pa B_1(0)} \bar u_{x_0,r}^2} \\
\Lambda_2(x,r)&= \frac{r^2 \int_{\pa B_r(x_0)} |\nabla_\theta v|^2+u^2 v^2 }{\int_{\pa B_r(x_0)} v^2} = \frac{ \int_{\pa B_1(0)} |\nabla_\theta \bar v_{x_0,r}|^2+r^2 \bar u_{x_0,r}^2 \bar v_{x_0,r}^2 }{\int_{\pa B_1(0)} \bar v_{x_0,r}^2},
\end{align*}
and $|\nabla_\theta u|^2= |\nabla u|^2- (\pa_\nu u)^2$.
\paragraph{Step 1)} \emph{There exist $\tilde C_1,\tilde C_2>0$ such that 
\[
\tilde C_1 \le \frac{\int_{\pa B_1(0)} \bar u_{x_0,r}^2}{\int_{\pa B_1(0)} \bar v_{x_0,r}^2}\le \tilde C_2 
\]
for every $x_0 \in\{|u-v| <\delta\}$ and $r \ge 1$}.\\
By contradiction, there are sequences $(x_i) \subset \{|u-v|<\delta\}$ and $(r_i) \subset [1,+\infty)$ such that
\[
\lim_{i \to \infty} \frac{\int_{\pa B_1(0)} \bar u_{x_i,r_i}^2}{\int_{\pa B_1(0)} \bar v_{x_i,r_i}^2}=+\infty
\]
(if the limit were $0$ we can argue in a similar way). By assumption \eqref{eq29}, we have
\[
\frac{\int_{\pa B_1(0)} \bar u_{x_i,r_i}^2}{\int_{\pa B_1(0)} \bar v_{x_i,r_i}^2} \le \frac{\int_{\pa B_1(0)} \bar u_{x_i,r_i}^2   }{C_1}.
\]
Consequently,  $\int_{\pa B_1(0)} \bar u_{x_i,r_i}^2 \to +\infty$ as $i \to \infty$, which in turns implies $H(x_i,r_i) \to +\infty$ as $i \to \infty$. Note that
\begin{equation}\label{eq2inACF}
\frac{\int_{\pa B_1(0)} \bar u_{x_i,r_i}^2}{\int_{\pa B_1(0)} \bar v_{x_i,r_i}^2} = \frac{\int_{\pa B_1(0)} u_{x_i,r_i}^2}{\int_{\pa B_1(0)} v_{x_i,r_i}^2} \quad \Rightarrow \quad  \lim_{i \to \infty} \frac{\int_{\pa B_1(0)} u_{x_i,r_i}^2}{\int_{\pa B_1(0)} v_{x_i,r_i}^2} =+\infty
\end{equation}
where we recall that the notation $(u_{x,r},v_{x,r})$ has been introduced in \eqref{def blow-down}. We set $(u_i,v_i):=(u_{x_i,r_i}, v_{x_i,r_i})$. By definition
\[
\begin{cases}
-\Delta u_i= -H(x_i,r_i) r_i^2 u_i v_i^2 & \text{in $\R^N$}\\
-\Delta v_i= -H(x_i,r_i) r_i^2 u_i^2 v_i  & \text{in $\R^N$}. 
\end{cases}
\]
and
\begin{equation}\label{eq3inACF}
\int_{\pa B_1(0)} u_i^2+v_i^2 = 1,
\end{equation}
which, by means of Corollary \ref{doubling}, provides a uniform-in-$i$ bound on $\int_{\pa B_r(0)} u_i^2+v_i^2$ for every $r \ge 1$. In light of the subharmonicity of $(u_i,v_i)$ this yields a uniform-in-$i$ bound on the $L^\infty$ norm of $\{(u_i,v_i)\}$ in every compact set of $\R^N$. As the competition parameter tends to $+infty$, we are in position to apply the local segregation Theorem \ref{segregation thm}, deducing that up to a subsequence $(u_i,v_i) \to (u_\infty,v_\infty)$ in $\mathcal{C}^0_{loc}(\R^N)$, where $u_\infty-v_\infty$ is harmonic and both $u_\infty$ and $v_\infty$ are subharmonic. By \eqref{eq2inACF}
\[
\int_{\pa B_1(0)} v_\infty^2=   \lim_{i \to +\infty} \int_{\pa B_1(0)} v_i^2 = \lim_{i \to \infty}  \frac{\int_{\pa B_1(0)} v_i^2   }{\int_{\pa B_1(0)} u_i^2+v_i^2 }=0.
\]
As $v_\infty$ is subharmonic and nonnegative, $v_\infty \equiv 0$. This implies that $u_\infty$ is harmonic and nonnegative in $B_1(0)$. Also, from \eqref{eq3inACF} it follows $\int_{\pa B_1(0)} u_\infty^2=1$. On the other hand, since $x_i \in \{|u-v|<\delta\}$ and $H(x_i,r_i) \to +\infty$ it results 
\[
u_\infty(0)=\lim_{i \to \infty}  \left( \frac{1}{\sqrt{H(x_i,r_i)}} |u(x_i)-v(x_i)| + v_i(0) \right) = 0  
\]
and by the strong maximum principle we obtain $u_\infty \equiv 0$, a contradiction.
\paragraph{Step 2)} \emph{Conclusion of the proof}.\\
 For $x_0 \in \{|u-v|<\bar \delta \}$ and $r \ge 1$, we consider the functions
\[
\tilde u_{x_0,r}(y):= \frac{\bar u_{x_0,r}(y)}{\left(\int_{\pa B_1(0)} \bar u_{x_0,r}^2\right)^{\frac{1}{2}} } \quad \text{and} \quad \tilde v_{x_0,r}(y):= \frac{\bar v_{x_0,r}(y)}{\left(\int_{\pa B_1(0)} \bar u_{x_0,r}^2\right)^{\frac{1}{2}} },
\] 
which are obtained by $\bar u_{x_0,r}$ and $\bar v_{x_0,r}$ after a normalization with respect to the $L^2$ norm of $\bar u_{x_0,r}$ on $\pa B_1(0)$. In light of assumption \eqref{eq29}
\begin{align*}
\Lambda_1(x_0,r) &= \int_{\pa B_1(0)} |\nabla_\theta \tilde u_{x_0,r}|^2 + r^2 \left(\int_{\pa B_1(0)} \bar u_{x_0,r}^2\right) \tilde{u}_{x_0,r}^2 \tilde{v}_{x_0,r}^2 \ge \int_{\pa B_1(0)} |\nabla_\theta \tilde u_{x_0,r}|^2 + C_1 r^2 \tilde{u}_{x_0,r}^2 \tilde{v}_{x_0,r}^2 \\
\Lambda_1(x_0,r) &= \int_{\pa B_1(0)} |\nabla_\theta \tilde v_{x_0,r}|^2 + r^2 \left(\int_{\pa B_1(0)} \bar u_{x_0,r}^2\right) \tilde{u}_{x_0,r}^2 \tilde{v}_{x_0,r}^2 \ge \int_{\pa B_1(0)} |\nabla_\theta \tilde v_{x_0,r}|^2 + C_1 r^2 \tilde{u}_{x_0,r}^2 \tilde{v}_{x_0,r}^2.
\end{align*}
As $\Gamma$ is monotone nondecreasing, we deduce
\begin{multline*}
\Gamma \left( \Lambda_1(x_0,r) \right) + \Gamma \left( \Lambda_2(x_0,r) \right)  \\
\ge \Gamma\left(\int_{\pa B_1(0)} |\nabla_\theta \tilde u_{x_0,r}|^2 + C_1 r^2 \tilde{u}_{x_0,r}^2 \tilde{v}_{x_0,r}^2\right) + \Gamma\left(\int_{\pa B_1(0)} |\nabla_\theta \tilde v_{x_0,r}|^2 + C_1 r^2 \tilde{u}_{x_0,r}^2 \tilde{v}_{x_0,r}^2 \right) .
\end{multline*}
Thanks to the first step, we are in position to apply Lemma 4.2 in \cite{Wa} in order to obtain
\[
\Gamma \left( \Lambda_1(x_0,r) \right) + \Gamma \left( \Lambda_2(x_0,r) \right) \ge 2- \frac{C}{ r^{\frac{1}{2}}},
\]
where $C$ is a positive constant independent on $x_0 \in \{|u-v|<\delta \}$ and $r \ge 1$. Coming back to \eqref{eq in ACF}, we deduce that there exists $C>0$ such that
\[
\frac{d}{dr}\log J(x_0,r) \ge - C  r^{-\frac{3}{2}}
\]
for every $x_0 \in \{|u-v|< \delta\}$, for every $r \ge 1$. An integration gives the desired result.
\end{proof}

In light of Lemma \ref{Wa corretto}, if $(u,v)$ is a solution of \eqref{system} having linear growth then Proposition \ref{corol 4.8} holds true. By means of this uniform monotonicity formula, we deduce the following statement.

\begin{corollary}\label{corol 4.10 in Wa}
Let $(u,v)$ be a solution of \eqref{system} having linear growth. Then there exists $\bar C_4>0$ such that
\begin{equation}\label{eq30}
\frac{1}{\bar C_4} \le J(x_0,r) \le \bar C_4, \qquad \int_{\pa B_1(x_0)} u^2+v^2 \le \bar C_4,
\end{equation}
and 
\[
\sup_{x \in B_R(x_0)} u(x)+v(x) \le \bar C_4(1+R)
\]
for every $x_0 \in \{|u-v|< \bar C_3\}$ and $r \ge 1$ (where $\bar C_3$ has been defined Remark \ref{def di bar C_3}).
\end{corollary}
\begin{proof}
In light of Proposition \ref{corol 4.8}, it is possible ti adapt the proof of Corollary 4.9 in \cite{Wa} (see also the discussion at the end of the proof) replacing $L_i$ with $\int_{\pa B_1(x_i)} u^2 +v^2$, where for us $(x_i) \subset \{|u-v| < \bar C_3\}$. In the quoted statement it is used the fact that $u(x_i)=v(x_i)$. As in this case $u(x_i) \neq v(x_i)$ in general, we obtain a contradiction with the same argument already used in the proof of Lemma \ref{Wa corretto}. This permits to deduce the existence of $\bar C_4>0$ such that \eqref{eq30} holds. Now, Corollary \ref{doubling} and the subharmonicity of $u$ and $v$ permits to obtain also the pointwise estimate of the thesis.
\end{proof}

\section{Uniqueness of the asymptotic profile}\label{sec: uniqueness}
 
In this section we show that, under assumptions \eqref{alg growth} and \eqref{uniform limit} (in fact it is sufficient to assume much less), any solution to \eqref{system} having algebraic growth is a solution with linear growth. Moreover, we will show that for every $x_0 \in \R^N$, the \emph{entire} blow-down family $\{(u_{x_0,R}, v_{x_0,R}): R >0\}$ converges, as $R \to +\infty$, to the same harmonic function. 

\begin{proposition}\label{uniqueness asymptotic profile thm}
Let $(u,v)$ be a solution of \eqref{system} satisfying assumptions \eqref{alg growth} and such that
\begin{equation}\label{v 0}
\lim_{x_N \to +\infty} v(x',x_N)= 0 \quad \text{uniformly in $x' \in \R^{N-1}$}.
\end{equation}
Then $N(x_0,r) \le 1$ for every $r>0$, and consequently $(u,v)$ has linear growth. Furthermore, there exists a constant $\gamma>0$ such that, for every $x_0 \in \R^N$, the blow-down family $\{(u_{x_0,R}, v_{x_0,R}): R >0\}$ converges to the pair $(\g x_N^+, \g x_N^-)$ as $R \to +\infty$, in $\mathcal{C}^0_{loc}(\R^N)$ and in $H^1_{loc}(\R^N)$.
\end{proposition}

\begin{remark}
It is possible to replace assumption \eqref{v 0} with 
\[
\lim_{x_N \to -\infty} u(x',x_N)= 0 \quad \text{uniformly in $x' \in \R^{N-1}$}.
\]
\end{remark}

\begin{proof}
As $(u,v)$ has algebraic growth, thanks to Lemma \ref{from growth to N} Theorem \ref{thm blow down} applies: for every $x_0 \in \R^N$ there exists
\[
\lim_{r \to +\infty} N(x_0,r)= d_{x_0} \in \N \setminus \{0\},
\]
and there exists a subsequence $(u_{x_0,R_n}, v_{x_0,R_n})$ of the blow-down family which is convergent ( in $\mathcal{C}^0_{loc}(\R^N)$ and in $H^1_{loc}(\R^N)$) to $(\Psi_{x_0}^+,\Psi_{x_0}^-)$, where $\Psi_{x_0}$ is a homogeneous harmonic polynomial of degree $d_{x_0} \ge 1$. As showed in Corollary \ref{cor blow-down}, this implies that \(\lim_{r \to \infty} H(x_0,r) = + \infty\). \\
Now, let $K \subset \subset \R^N_+$. Since
\[
\inf\{x_N: x \in K\} > 0,
\]
in light of assumption \eqref{v 0} there holds
\[
\lim_{R \to +\infty} v_R(x) = 0  \quad \text{uniformly in $K$}.
\]
As $K$ has been arbitrarily chosen, it follows that $v_{x_0,R_n}(x) \to 0$ pointwise in $\R^N_+$. By the uniqueness of the limit, we deduce $\Psi_{x_0}^-=0$ in $\R^N_+$. Thus, $\Psi_{x_0}$ is an homogeneous harmonic polynomial (hence $\Psi_{x_0}(0)=0$) which is nonnegative in $\R^N_+$ and is not identically $0$ (this follows simply from the fact that $d_{x_0} \ge 1$):
\[
\begin{cases}
-\Delta \Psi_{x_0}=0 & \text{in $\R_+^N$} \\
\Psi_{x_0}\ge 0, \ \Psi_{x_0} \not \equiv 0 & \text{in $\R_+^N$} \\
\Psi_{x_0}(0)=0.
\end{cases} 
\]
By the strong maximum principle, we deduce that $\Psi_{x_0}>0$ in $\R^N_+$; hence, the Hopf' Lemma guarantees that $\nabla \Psi_{x_0}(0) \neq 0$. The unique (up to a constant factor) homogeneous harmonic polynomial satisfying these properties is the linear one: $\Psi_{x_0}(x)= C_{x_0} x_N$; but $C_{x_0}>0$ is uniquely determined (independently on $x_0$) by the condition
\[
\int_{\pa B_1(0)} C_{x_0}^2 x_N^2 = \lim_{n \to \infty} \int_{\pa B_1(0)} u_{x_0,R_n}^2 + v_{x_0,R_n}^2 =   1.
\]
Hence, for every $x_0$ the blow-down family converges (up to a subsequence) to the same pair $(\gamma x_N^+,\g x_N^-)$, for a constant $\g>0$. By Theorem \ref{thm blow down}, the fact that the degree of the limiting profile is $1$ means that $d_{x_0}=1$ for every $x_0 \in \R^N$, and this gives the linear growth of $(u,v)$, see Corollary \ref{from N to growth}. \\
It remains to show that, for every $x_0 \in \R^N$, the entire blow-down family converges to $\gamma x_N$. Assume by contradiction that this is not true: there exist a compact $K \subset \R^N$, a $\bar \eps>0$ and a subsequence $\{(u_{x_0,R_m}, v_{x_0,R_m})\}$ with $R_m \to +\infty$ as $m \to \infty$, such that
\begin{multline}\label{entire convergence}
\| u_{x_0,R_m} - \gamma x_N^+\|_{\mathcal{C}^0(K)} + \|u_{x_0,R_m}- \gamma x_N^+\|_{H^1(K)} \\
+ \| v_{x_0,R_m} - \gamma x_N^-\|_{\mathcal{C}^0(K)} + \|v_{x_0,R_m}- \gamma x_N^-\|_{H^1(K)} \ge \bar \eps
\end{multline}
for every $m$. But now it is possible to repeat step by step the proof of Theorem \ref{thm blow down} obtaining that, up to a subsequence, $\{(u_{x_0,R_m}, v_{x_0,R_m})\}$ converges, as $m \to +\infty$ to a homogeneous harmonic polynomial of degree $d_{x_0} \ge 1$. Following the above line of reasoning, we find that the limit is nothing but the function $(\gamma x_N^+,\gamma x_N^-)$, in contradiction with \eqref{entire convergence}.
\end{proof}

\section{Monotonicity at infinity}\label{sec: monot at infinity}

We aim at proving the following statement.
\begin{proposition}\label{monotonicity at infinity}
Let $(u,v)$ be a solution of \eqref{system} satisfying \eqref{alg growth} and \eqref{uniform limit}. For every 
\[
\nu \in \{ \nu \in \S^{N-1}: \langle \nu, e_N \rangle >0 \},
\]
there exists $M_\nu>0$ such that
\[
x \in \{x_N>M_\nu\} \Rightarrow \pa_\nu u(x)>0 \quad \text{and} \quad x \in\{x_N<-M_\nu\} \Rightarrow \pa_\nu v(x)<0.
\]
\end{proposition}

The achievement of section \ref{sec: uniqueness} says that $(u,v)$ behaves at infinity as $(\gamma x_N^+, \gamma x_N^-)$; thus, the idea is that $u$ has to be increasing in the $e_N$ direction for $x_N \gg 1$ and $v$ has to be decreasing in the $e_N$ direction for $x_N \ll -1$. In order to prove this conjecture, we wish to apply the standard gradient estimate for the Poisson equation (see e.g. \cite{GiTr}) on $u$ minus "a suitable linear function": this idea is corroborated by the fact that $\Delta u$ can be uniformly bounded by an exponentially decaying function for $x_N$ sufficiently large. An analogous bound holds for $\Delta v$ when $x_N$ is sufficiently large and negative.

\begin{lemma}\label{estimate on the Delta}
Let $(u,v)$ be a solution of \eqref{system} satisfying \eqref{uniform limit}. For every $p,q \ge 1$ there exist $M_1(p,q)>0$ and a positive constant $C=C(p,q)>0$ such that 
\[
u^p(x) v^q(x) \le C e^{-C |x_N|} \qquad \forall x \in \{|x_N|>M_1(p,q)\}.
\]
\end{lemma}

\begin{proof}
We consider the bound on $u^p v^q$ in $x_N \gg 1$, the same argument applies for $x_N \ll -1$.\\
Given $K>0$ and $\delta>0$, by \eqref{uniform limit} there exists $M>0$ such that
\[
u(x)>K \quad \text{and} \quad v(x)<\delta \quad \text{if $x \in \{x_N>M/2\}$}.
\]
For every $x \in \{x_N> M\}$ the ball $B_x := B_{x_{N}/4}(x)$ is contained in $\{x_N>M/2\}$. Consequently,
\[
\begin{cases}
u(y) \ge K_x:= \inf_{B_x} u \ge K \\
v(y) \le \delta 
\end{cases} \qquad \forall y \in B_x, \ \forall x \in \{x_N>M\},
\]
so that 
\[
\begin{cases}
-\Delta v \le -K_x^2 v & \text{in $B_x$} \\
v \ge 0 & \text{in $B_x$}  \\
v \le \delta & \text{in $B_x$} .
\end{cases}
\]
We are in position to apply Lemma \ref{lemma exponential decay}: 
\begin{equation}\label{eq5}
\sup_{B_x'} v \le C \delta e^{- C K_x x_N},
\end{equation}
where $B_x'$ denotes the ball $B_{x_N/8}(x)$. On the other hand, it is possible to apply the Harnack inequality (Theorem 8.20 in \cite{GiTr}, see also the subsequent observation concerning the estimate on the constant) on $u$ in $B_x$, with potential $v^2$:
\begin{equation}\label{eq6}
\sup_{B_x} u \le C e^{C \delta x_N} K_x.
\end{equation}
The inequalities \eqref{eq5} and \eqref{eq6} yields
\[
u^p(x) v^q(x) \le C K_x^p \delta^q e^{-C_1 q K_x x_N + C_2 p \delta x_N} \qquad \forall x \in \{x_N>M\}.
\]
A suitable choice of $K \le K_x$ and $\delta$ permits to obtain the desired result.
\end{proof}

\begin{remark}\label{rem su corol 1}
From now on we will denote as $M_1:= \max\{ M_1(1,2), M_1(2,1)\}$, where $M_1(1,2)$ and $M_1(2,1)$ have been defined in Lemma \ref{estimate on the Delta}.
\end{remark}

If we could show that the function $u$ can be approximated in $\{x_N>M_1\}$ by a linear function with positive slope in the $e_N$ direction, the gradient estimates for the Poisson equation would give the desired monotonicity for $u$. So far we showed that for given $x_0 \in \R^N$ and $\eps>0$ there exists $R_{x_0,\eps}>0$ such that
\begin{equation}\label{blow-down base}
\sup_{x \in B_1(0)} |u_{x_0,R}(x)- \gamma x_N^+| + |v_{x_0,R}(x)- \gamma x_N^-| < \eps
\end{equation}
for every $R>R_{x_0,\eps}$. This means that
\[
\sup_{x \in B_R(x_0)} \left|u(x)- \gamma \frac{\sqrt{H(x_0,R)}}{R} (x_N-x_{0,N})^+\right| + \left|v(x)- \gamma \frac{\sqrt{H(x_0,R)}}{R} (x_N-x_{0,N})^- \right| < \sqrt{H(x_0,R)}\eps
\]
whenever $R>R_{x_0,\eps}$. This reveals that we have to face two problems: the first one is the fact that we have not a unique candidate to approximate $u$ for $x_N \gg 1$ and $v$ for $x_N \ll -1$, the second one is that this approximation, which holds for $R$ sufficiently large, get worse as $R$ increases (recall that the function $H(x_0,\cdot)$ is nondecreasing and tends to $+\infty$ as $R \to +\infty$, see Corollary \ref{cor blow-down}). In order to overcome the second problem, we wish to find a uniform estimate (in both $x_0$ and $R$) on the ratio $\frac{\sqrt{H(x_0,R)}}{R}$; in the forthcoming Lemma \ref{control on H/R}, we show that this is possible if $x_0 \in \{|u-v|< \bar C_3\}$, where $\bar C_3$ has been defined in Remark \ref{def di bar C_3}. Before, we deduce some useful information about this special set.

\begin{lemma}\label{about u=v}
Under the assumption \eqref{uniform limit}, the set $\{|u-v|< \bar C_3\}$ is bounded in the $e_N$ direction and unbounded in all the other directions $\{e_1,\ldots, e_{N-1}\}$. In particular, for every $x' \in \R^{N-1}$ there exists $\tilde x \in \{|u-v|<\bar C_3\}$ such that $\tilde x'=x'$. 
\end{lemma}

\begin{proof}
The properties follow easily by our main assumption \eqref{uniform limit}. Indeed, by considering the function $u-v$ one sees that 
\[
\lim_{x_N \to \pm\infty} (u(x',x_N)-v(x',x_N))= \pm \infty,
\]
uniformly in $x' \in \R^{N-1}$. This immediately implies that the level set $\{|u-v| \le M\}$ is bounded in the $e_N$ direction for every $M>0$ (in particular, this holds for $\bar C_3$). On the other hand, for a given $x' \in \R^{N-1}$ we can consider the map $s \in \R \mapsto u(x',s)-v(x',s)$. This is a continuous function which tends to $\pm \infty$ as $s \to \pm \infty$, thus there exist $\tilde s \in \R$ such that $|u(x',\tilde s)-v(x',\tilde s)|<\bar C_3$. 
\end{proof}

\begin{remark}\label{def zeta}
From now on, we denote $\zeta:= \sup\{|x_{0,N}|: x_0 \in \{|u-v|< \bar C_3\} \}<+\infty$.
\end{remark}

In the next Lemma we give uniform upper and lower bounds on the ratio $\frac{\sqrt{H(x_0,R)}}{R}$ for $x_0 \in \{|u-v|<\bar C_3\}$ and $R \ge 1$.
 
\begin{lemma}\label{control on H/R}
Let $(u,v)$ be a solution of \eqref{system} satisfying \eqref{alg growth} and \eqref{uniform limit}. There exists $\bar C_5,\bar C_6>0$ such that
\[
\bar C_5 \le \frac{\sqrt{H(x_0,R)}}{R} \le \bar C_6 
\]
for every $x_0 \in \{|u-v|< \bar C_3\}$ and $R \ge 1$.
\end{lemma}

\begin{proof}
By Proposition \ref{uniqueness asymptotic profile thm}, we know that under \eqref{alg growth} and \eqref{uniform limit} the solution $(u,v)$ has linear growth. Hence, we can invoke Corollary \ref{corol 4.10 in Wa}; combining this result with Corollary \ref{doubling} we deduce
\[
\frac{H(x_0,R)}{R^2} \le e H(x_0,1) \le e \bar C_4 \qquad \forall x_0 \in \{|u-v|<\bar C_3\}, \ R \ge 1.
\]
For the lower bound, we show that the quantity
\begin{multline*}
J_{x_0,R}(0,1):= \int_{B_1(0)} \frac{|\nabla u_{x_0,R}(y)|^2 + H(x_0,R)R^2 u_{x_0,R}^2(y) v_{x_0,R}^2(y)}{|y|^{N-2}}\,dy \\
\cdot \int_{B_1(0)} \frac{|\nabla v_{x_0,R}(y)|^2 + H(x_0,R)R^2 u_{x_0,R}^2(y) v_{x_0,R}^2(y)}{|y|^{N-2}}\,dy
\end{multline*}
is bounded above by a positive constant $C$ independent on $x_0 \in \R^N$ and $R \ge 1$. We use the \eqref{4.11 by Wang}: there exists $C>0$ independent on $x_0 \in \R^N$ and on $R \ge 1$ such that
\begin{multline}\label{eq1}
\int_{B_1(0)} \frac{|\nabla u_{x_0,R}(y)|^2 + H(x_0,R)R^2 u_{x_0,R}^2(y) v_{x_0,R}^2(y)}{|y|^{N-2}}\,dy = \frac{1}{H(x_0,R)} \int_{B_R(x_0)} \frac{|\nabla u(y)|^2+ u^2(y) v^2(y)}{|y-x_0|^{N-2}}\,dy \\
 \le \frac{C}{H(x_0,R) R^N} \int_{B_{2R}(x_0)} u^2 = C \int_{B_2(0)} u_{x_0,R}^2.
\end{multline}
We point out that, as $N(x_0,r) \le 1$ for every $x_0 \in \R^N$ and $r \ge 1$, the same estimate holds true for the Almgren quotient associated to $(u_{x_0,R},v_{x_0,R})$, for every $x_0 \in \R^N$ and $R \ge 1$ (see Remark \ref{N ed N riscalato}). As a consequence, the normalization $\int_{\pa B_1(0)} u_{x_0,R}^2 + v_{x_0,R}^2 =1$ gives, by Corollary \ref{doubling}, a uniform (in both $x_0$ and $R$) upper bound for $\int_{\pa B_3(0)} u_{x_0,R}^2+ v_{x_0,R}^2$. Due to the subharmonicity of $(u_{x_0,R}, v_{x_0,R})$, we obtain a uniform bound for $\{(u_{x_0,R},v_{x_0,R})\}$ in $L^\infty(B_{2}(0))$, so that we can estimate the right hand side of \eqref{eq1} obtaining
\[
\int_{B_1(0)} \frac{|\nabla u_{x_0,R}(y)|^2 + H(x_0,R)R^2 u_{x_0,R}^2(y) v_{x_0,R}^2(y)}{|y|^{N-2}}\,dy \le C
\]
for every $x_0 \in \R^N$ and $R \ge 1$. Arguing in the same way on the second factor of $J_{x_0,R}(0,1)$ we obtain the desired upper bound: there exists $C>0$ such that
\[
J_{x_0,R}(0,1) \le C \qquad \forall x_0 \in \R^N, \ \forall R \ge 1.
\]
A simple change of variable shows that $J_{x_0,R}(0,1) = \frac{R^4}{H^2(x_0,R)}   J(x_0,R)$, so that 
\begin{equation}\label{eq3}
  J(x_0,R) \le  C \frac{H^2(x_0,R)}{R^4}  \qquad \forall x_0 \in \R^N, \ \forall R \ge 1.
\end{equation}
A comparison between \eqref{eq3} and the uniform lower estimate of Corollary \ref{corol 4.10 in Wa} provides the desired result: 
\[
\frac{H^2(x_0,R)}{R^4} \ge \frac{C}{\bar C_4} \qquad \forall x_0 \in \{|u-v| < \bar C_3\}, \ \forall R \ge 1. \qedhere
\]
\end{proof}

We are ready to improve the estimate given by \eqref{blow-down base}. Firstly, we get rid of the dependence of $R_{x_0,\eps}$ on $x_0$ for $x_0 \in \{|u-v|<\bar C_3\}$.

\begin{lemma}\label{get rid of x_0}
Let $(u,v)$ be a solution of \eqref{system} satisfying \eqref{alg growth} and \eqref{uniform limit}. For every $\eps>0$ there exists $R_\eps>0$ such that 
\[
\sup_{x \in B_1(0)} |u_{x_0,R}(x)- \gamma x_N^+| + |v_{x_0,R}(x)- \gamma x_N^-| < \eps
\]
for every $R>R_\eps$ and $x_0 \in \{|u-v| < \bar C_3\}$, where $\gamma$ and $\bar C_3$ have been defined in Proposition \ref{uniqueness asymptotic profile thm} and Remark \ref{def di bar C_3} respectively.
\end{lemma}

\begin{proof}
Assume by contradiction that there exist $\bar \eps>0$ and a sequence $(x_j,R_j)$ with $x_j \in \{|u-v| < \bar C_3\}$ for every $j$, $R_j \to +\infty$, and
\begin{equation}\label{absurd j}
\sup_{x \in B_1(0)} |u_{x_j,R_j}(x)- \gamma x_N^+| + |v_{x_j,R_j}(x)- \gamma x_N^-| \ge \bar\eps
\end{equation}
for every $j$. Let us denote $(u_j,v_j)=(u_{x_j,R_j}, v_{x_j,R_j})$. We know that $(u_j,v_j)$ solves
\[
\begin{cases}
-\Delta u_j = - H(x_j,R_j) R_j^2 u_j v_j^2 & \text{in $\R^N$} \\
-\Delta v_j = - H(x_j,R_j) R_j^2 u_j^2 v_j & \text{in $\R^N$} 
\end{cases} \qquad \forall j.
\]
In light of Lemma \ref{control on H/R}, we know that
\begin{equation}\label{eq31}
\lim_{j \to +\infty} H(x_j, R_j) \ge \lim_{j \to +\infty} \bar C_5 R_j^2 = +\infty;
\end{equation}
a fortiori the competition parameter $H(x_j,R_j) R_j^2$ tends to $+\infty$ as $j \to +\infty$. Note that the normalization $\int_{\pa B_1(0)} u_j^2 + v_j^2=1$ implies, by means of Corollary \ref{doubling} (which we can apply on $(u_j,v_j)$, see Remark \ref{N ed N riscalato}), that
\[
\int_{\pa B_r(0)} u_j^2 + v_j^2 \le e r^{N+1} \qquad \forall r > 1, \ \forall j.
\]
By subharmonicity, the sequence $\{(u_j,v_j)\}$ is uniformly bounded in every compact set $K$ of $\R^N$, and in light of Theorem \ref{local Holder estimate} it is also uniformly bounded in $\mathcal{C}^{0,\alpha}(K)$, for every $\a \in (0,1)$. The local segregation Theorem \ref{segregation thm} implies that, up to a subsequence, $(u_j,v_j) \to (u_\infty,v_\infty)$ in $\mathcal{C}^0_{loc}(\R^N) \cap H^1_{loc}(\R^N)$, and
\begin{itemize}
\item[($i$)] $u_\infty v_\infty \equiv 0$ in $\R^N$,
\item[($ii$)] $H(x_j, R_j) R_j^2 u_j^2 v_j^2 \to 0$ as $j \to \infty$ in $L^1_{loc}(\R^N)$,
\item[($iii$)] $u_\infty-v_\infty$ is harmonic in $\R^N$,
\item[($iv$)] by \eqref{eq31} and the fact that $x_j \in \{|u-v|<\bar C_3\}$
\[
|u_\infty(0)-v_\infty(0)|= \lim_{j \to +\infty} \frac{1}{\sqrt{H(x_j,R_j)}} |u(x_j)-v(x_j)|= 0
\]
\item[($v$)] by uniform convergence the normalization on $\pa B_1(0)$ pass to the limit:
\begin{equation}\label{normalization infty} 
\int_{\pa B_1(0)} u_\infty^2 + v_\infty^2 = 1,
\end{equation}
\item[($vi$)] by $H^1$ and uniform convergence and the point ($ii$)
\begin{multline}\label{uniform upper bound N_infty}
\frac{r \int_{B_r(0)} |\nabla u_\infty|^2 + |\nabla v_\infty|^2}{\int_{\pa B_r(0)} u_\infty^2 + v_\infty^2} = \lim_{j \to +\infty} \frac{r \int_{B_r(0)} |\nabla u_j|^2 + |\nabla v_j|^2 + H(x_j, R_j) R_j^2 u_j^2 v_j^2}{\int_{\pa B_r(0)} u_j^2 + v_j^2} \\
= \lim_{j \to +\infty} N(x_j, R_j r) \le 1 \qquad \forall r \in (0,1),
\end{multline}
where the upper bound on $N$ follows from the fact that, under assumptions \eqref{alg growth} and \eqref{uniform limit}, Proposition \ref{uniqueness asymptotic profile thm} applies and guarantees that $(u,v)$ has linear growth.
\end{itemize}  
Note that 
\[
N_\infty(0,r):=\frac{r \int_{B_r(0)} |\nabla u_\infty|^2 + |\nabla v_\infty|^2}{\int_{\pa B_r(0)} u_\infty^2 + v_\infty^2} 
\]
is the Almgren quotient of the harmonic function $u_\infty-v_\infty$, and it is nondecreasing. As $u_\infty(0)-v_\infty(0)=0$, it results
\begin{equation}\label{uniform lower bound N_infty}
N_\infty(0,r) \ge \lim_{s \to 0^+} N_\infty(0,s) = \deg(u_\infty-v_\infty,0) \ge 1
\end{equation}
for every $r>0$. Here, $\deg(u_\infty-v_\infty,0)$ denotes the degree of vanishing of the harmonic function $u_\infty-v_\infty$ in $0$, and is greater then $1$ because it has to be a positive integer (this result is by now well known). By monotonicity, a comparison between \eqref{uniform upper bound N_infty} and \eqref{uniform lower bound N_infty} yields $N_\infty(0,r)  = 1$ for every $r \in (0,1)$, which implies (see Proposition 3.9 in \cite{NoTaTeVe}, which we can apply, as explained in Remark \ref{rem su CPAM e dim}) that $u_\infty-v_\infty$ is a linear function, that is, $(u_\infty(x),v_\infty(x))= (\langle e, x\rangle^+, \langle e, x\rangle^-)$ for some $e \in \R^N$. We claim that 
\begin{equation}\label{e= gamma e_N}
e= \gamma e_N,
\end{equation}
 which gives a contradiction with \eqref{absurd j} and completes the proof of the statement. To prove the claim, we note that under our assumptions we have
 \[
 v_j(x)= \frac{1}{\sqrt{H(x_j,R_j)}} v(x_j'+ R_j x', x_{j,N}+R_j x_N) \to 0 
 \]
 as $j \to +\infty$, uniformly in every compact subset of $B_1(0) \cap \R^N_+$; to pass to the limit, we used the fact that $H(x_j,R_j) \ge \bar C_1$ (see Lemma \ref{uniform growth for H near u=v}) and the boundedness of the set $\{|u-v|< \bar C_3\}$ in the $e_N$ direction (see Lemma \ref{about u=v}), which guarantees that $x_{j,N} + R_j x_N \to +\infty$ as $j \to +\infty$. By the uniqueness of the limit, we deduce $e= C e_N$ for some $C>0$. The normalization \eqref{normalization infty} yields $C=\gamma$, which concludes the proof of the claim \eqref{e= gamma e_N}.
 \end{proof}

\begin{definition}
Let us fix $\tau>0$ not too small (to be determined in the following Lemma). For a given $x_0 \in \R^N$ and $R>0$ we introduce the conical sectors
\begin{gather*}
S_{x_0,R}^+:= \left\{x=(x',x_N) \in \R^N: \frac{R}{2} < |x-x_0| <R, |x'-x_0'|< \tau (x_N-x_{0,N}) \right\} \\
S_{x_0,R}^-:= \left\{x=(x',x_N) \in \R^N: \frac{R}{2} < |x-x_0| <R, |x'-x_0'|< \tau (x_{0,N}-x_N) \right\},
\end{gather*}
and their union $S_{x_0,R}$.
\end{definition}

The following picture represents the set $S_{x_0,R}^+$ for a given $x_0 \in \R^N$.

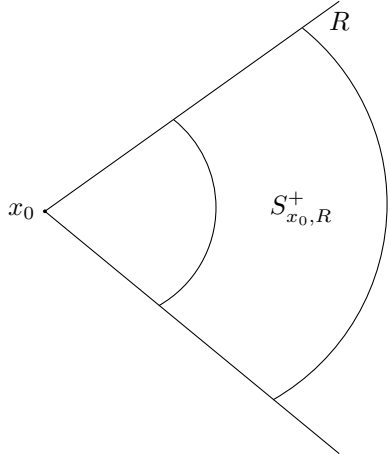
\begin{center}
\begin{tikzpicture}[scale=0.5]
\draw (59:0.5) node[anchor=east]{$x_0$};
\filldraw (59:0.5) circle (1pt);
\draw (59:0.5)+(3cm,-2.5cm) arc (-60:51:3cm);
\draw (59:0.5)+(6cm,-5cm) arc (-60:51:6cm);
\draw (59:0.5) -- (8cm,6cm);
\draw (8cm,5cm) node[anchor=south]{$R$};
\draw (59:0.5) -- (8cm,-6cm);
\draw (7cm,0.5cm) node{$S_{x_0,R}^+$};
\end{tikzpicture}
\end{center}

The geometry of the set $\{|u-v|<\bar C_3\}$ allows to show that the union of $S_{x_0,R}$ with $R$ sufficiently large and $x_0 \in \{|u-v|< \bar C_3\}$ contains, and it is contained in, the union of two half-spaces. 

\begin{lemma}\label{union of the cones}
Let $(u,v)$ be a solution of \eqref{system} satisfying \eqref{alg growth} and \eqref{uniform limit}. There exists $\bar R>0$ such that, for every $\widehat R \ge \bar R$ there exists $M_2 = M_2(\widehat{R}) > \zeta$ such that
\[
\{|x_N|> M_2\} \subset \bigcup_{\substack{x_0 \in \{|u-v|<\bar C_3\} \\ R> \widehat R}} S_{x_0,R} \subset \{|x_N|>\zeta\},
\]
where $\zeta$ has been defined in Remark \ref{def zeta}. Furthermore, for every $N \ge 2$ we can choose $\tau>0$ such that, if $x \in \{|x_N|> M_2\}$, there exist $\tilde x \in \{|u-v|< \bar C_3\}$ and $\tilde R> \widehat R$ such that
\[
\overline{Q}_x \subset S_{\tilde x,\tilde R},
\]
where $Q_x$ denotes the open cube centered in $x$ with side $\displaystyle \frac{x_N}{100}$.
\end{lemma}

\begin{proof}
Thanks to Lemma \ref{about u=v}, it is not difficult to see that, provided $\bar R$ is sufficiently large and $\widehat{R}>\bar R$, it results 
\[
\bigcup_{\substack{x_0 \in \{|u-v|<\bar C_3\} \\ R> \widehat R}} S_{x_0,R}  \subset \bigcup_{\substack{x_0 \in \{|u-v|<\bar C_3\} \\ R> \bar R}} S_{x_0,R} \subset \{|x_N|>\zeta\}.
\] 
Now we argue in $\R^N_+$ showing that there exists $ M_2 = M_2(\widehat{R})> \zeta$ such that
\[
\{x_N> M_2\} \subset \bigcup_{\substack{x_0 \in \{|u-v|<\bar C_3\} \\ R> \widehat R}} S_{x_0,R}^+,
\]
and that for every $x \in \{x_N>M_2\}$ there exist the desired $\tilde x$ and $\tilde R$. For $x \gg 1$, let $\tilde x$ the point of $\{|u-v|<\bar C_3\}$ such that $\tilde x'=x'$ ($\tilde{x}$ exists, see Lemma \ref{about u=v}). Provided $\tau$ is not too small, the cube centered in $x$ with side $\frac{x_N}{100}$ is contained in the conical sector $S_{\tilde x, \tilde R}^+$ for $\tilde R:= \frac{3}{2}(x_N- \tilde x_N)$. Note that,
\[
\frac{3}{2}(x_N- \tilde x_N) \ge \frac{3}{2} (x_N-\zeta) \ge \frac{5}{4} x_N > \widehat R.
\]
whenever $x_N > M_2 := \max\left\{6 \zeta, \frac{4}{5} \widehat R\right\}$. The same argument works in the half-space $\R^N_-$.
\end{proof}

\begin{remark}\label{rem sui quadrati}
From the previous proof we see that, fixed $\widehat R>\bar R$, it is possible to associate to every $x \in \{|x_N| > M_2\}$ the conical sector $S_{\tilde x,\tilde R}$ which contains the cube $Q_x$; that is, $\tilde x$ is a point of $\{|u-v|< \bar C_3\}$ such that $\tilde x'=x'$ and 
\[
\tilde R= \begin{cases}
                \frac{3}{2}(x_N-\tilde x_N) & \text{if $x_N>  M_2$}\\
                \frac{3}{2}(\tilde x_N-x_N) & \text{if $x_N<- M_2$}.
                \end{cases}
                \]
\end{remark}

In each $S_{x_0,R}$ we can obtain a further improvement, by means of Lemma \ref{control on H/R}, of the estimates of Lemma \ref{get rid of x_0}. 

\begin{lemma}\label{uniform estimate in a cone}
Let $(u,v)$ be a solution of \eqref{system} satisfying \eqref{alg growth} and \eqref{uniform limit}. For every $\eps>0$, if $R>R_{\eps}$ and $x_0 \in \{|u-v|< \bar C_3\}$ then
\[
\sup_{x \in S_{x_0,R}} \left| \frac{u(x) -\gamma \frac{\sqrt{H(x_0,R)}}{R}(x_N-x_{0,N})^+}{|x-x_0|} \right|
+ \left| \frac{v(x) -\gamma \frac{\sqrt{H(x_0,R)}}{R}(x_N-x_{0,N})^-}{|x-x_0|} \right| < \eps,
\]
with $\bar C_5 \le \frac{\sqrt{H(x_0,R)}}{R} \le \bar C_6$. We recall that $\bar C_3, \bar C_5, \bar C_6$ and $R_\eps$ have been defined in Remark \ref{def di bar C_3}, Lemma \ref{control on H/R} and Lemma \ref{get rid of x_0} respectively.
\end{lemma}

\begin{proof}
Lemma \ref{get rid of x_0} ensures that for every $R>R_{\eps}$, for every $x_0 \in \{|u-v|< \bar C_3\}$
\[
\sup_{x \in S_{0,1}} \left| \frac{u(x_0+Rx)}{\sqrt{H(x_0,R)}}-\gamma x_N^+ \right| + \left| \frac{v(x_0+Rx)}{\sqrt{H(x_0,R)}}-\gamma x_N^- \right| < \eps,
\]
that is, 
\[
\left|u(x_0+Rx) -\gamma  \sqrt{H(x_0,R)} x_N^+\right| + \left|v(x_0+R x) - \gamma \sqrt{H(x_0,R)} x_N^- \right|< \sqrt{H(x_0,R)} \eps
\]
for every $x \in S_{0,1}$. Consequently, dividing both the sides for $R$ we obtain
\[
|x| \left(\left| \frac{u(x_0+Rx)}{|Rx|}-\gamma \frac{\sqrt{H(x_0,R)}}{R}\frac{R x_N^+}{|Rx|} \right|
+ \left| \frac{v(x_0+Rx)}{|Rx|}-\gamma \frac{\sqrt{H(x_0,R)}}{R}\frac{R x_N^-}{|Rx|} \right|\right) < \frac{\sqrt{H(x_0,R)}}{R}\eps
\]
for every $x \in S_{0,1}$, provided $R>R_\eps$ and $x_0 \in \{|u-v|< \bar C_3\}$. In turns, this gives 
\[
\sup_{x \in S_{x_0,R}} \left| \frac{u(x) -\gamma \frac{\sqrt{H(x_0,R)}}{R}(x_N-x_{0,N})^+}{|x-x_0|} \right|
+ \left| \frac{v(x) -\gamma \frac{\sqrt{H(x_0,R)}}{R}(x_N-x_{0,N})^-}{|x-x_0|} \right|
< 2\frac{\sqrt{H(x_0,R)}}{R}\eps
\]
for every $R>R_{\eps}$ and $x_0 \in \{|u-v|< \bar C_3\}$. Finally, we can use the upper bound on $\frac{\sqrt{H(x_0,R)}}{R}$, see Lemma \ref{control on H/R}.
\end{proof}

We are ready to apply the gradient estimates for the Poisson equation in a half-space $x_N \gg1$; we will show that if $x_N>0$ is sufficiently large then there exists a linear functions $\varphi_x$ (depending on $x$) which approximate $u$ in a $\mathcal{C}^1$-sense in $x$. In light of the uniform control given in Lemma \ref{control on H/R}, the slope of $\varphi_x$ will turn to be uniformly bounded from below in an entire half-space (the same holds for $v$ in $x_N \ll -1$), allowing to conclude the proof of Proposition \ref{monotonicity at infinity}. It is essential to work in conical sectors, because in this way we can control the quantity $|x-x_0|$ with the privileged component $|x_N-x_{0,N}|$.

\begin{lemma}\label{C^1 approx}
Let $(u,v)$ be a solution of \eqref{system} satisfying \eqref{alg growth} and \eqref{uniform limit}. For every $\eps>0$ there exists $M_{\eps}>0$ such that 
\[
\left|\nabla u(x) - \gamma \frac{\sqrt{H(\tilde x, \tilde R)}}{\tilde R}e_N \right| < \eps \qquad \forall x \in \{x_N>M_{\eps}\},
\]
where $\tilde x$ and $\tilde R$ have been defined in Remark \ref{rem sui quadrati}. Analogously, 
\[
\left|\nabla v(x) - \gamma \frac{\sqrt{H(\tilde x, \tilde R )}}{\tilde R}e_N \right| < \eps \qquad \forall x \in \{x_N<-M_{\eps}\}.
\]
\end{lemma}
\begin{proof}
For every $\eps>0$, let $R_\eps$ be defined in Lemma \ref{get rid of x_0}. Let $M_{2,\eps}:= M_2 (\max\{\bar R, R_\eps\})$, where $M_2$ has been defined in Lemma \ref{union of the cones}. Let $M_{\eps}:= \max\{M_1, M_{2,\eps}\}$, where $M_1$ has been defined in Remark \ref{rem su corol 1}. For $x \in \{x_N>M_{\eps}\}$, there are $\tilde R>R_{\eps}$ and $\tilde x \in \{|u-v|<\bar C_3\}$ such that $\overline{Q}_x \subset S_{\tilde x, \tilde R}^+$, see Lemma \ref{union of the cones} and Remark \ref{rem sui quadrati}. By the gradient estimates for the Poisson equation (see \cite{GiTr}, section 3.4) plus Lemmas \ref{estimate on the Delta} and \ref{uniform estimate in a cone}, we deduce that 
\begin{equation}\label{eq32}
\begin{split}
\left| \nabla u(x) - \gamma \frac{\sqrt{H(\tilde x, \tilde R)}}{\tilde R}e_N \right| & \le \frac{C}{x_N} \sup_{y \in \overline{Q}_x} \left|u(y) - \gamma \frac{\sqrt{H(\tilde x, \tilde R)}}{\tilde R}(y_N-\tilde x_{N}) \right| + \frac{x_N}{2} \sup_{y \in \overline{Q}_x}  v^2(y) u(y) \\
& \le \frac{C}{x_N} \sup_{y \in \overline{Q}_x} \eps |y - \tilde x| + C x_N e^{-C x_N} .
\end{split} 
\end{equation}
As $\overline{Q}_x \subset S_{\tilde x,\tilde R}^+$, for every $y \in \overline{Q}_x$ it results 
\begin{align*}
|y-\tilde x| & < (\tau +1 )(y_N-\tilde x_{N}) \le     (\tau +1) (y_N-x_N)+ (\tau +1) (x_N-\tilde x_N) \\
&\le C x_N + (\tau +1)(x_N +\zeta) \le Cx_N,
\end{align*}
where we recall that $\zeta= \sup\{x_{0,N}: x_0 \in \{u=v\}\}< M_\eps<x_N$. Plugging this estimate into the \eqref{eq32}, we obtain 
\[
\left| \nabla u(x) - \gamma \frac{\sqrt{H(\tilde x, \tilde R)}}{\tilde R}e_N \right| \le C\eps + C x_N e^{-C x_N}
\]
whenever $x_N>M_{\eps}$; if necessary, we can replace $M_{\eps}$ with a larger quantity, obtaining the thesis for $u$.\\
A similar argument can be carried on for $v$.
\end{proof}

\begin{proof}[Conclusion of the proof of Proposition \ref{monotonicity at infinity}]
Given $\nu \in \{\nu \in \S^{N-1}: \langle \nu, e_N \rangle >0\}$, we choose
\[
0<\eps(\nu) \le \frac{\gamma \bar C_5}{2} \langle e_N, \nu \rangle.
\]
where $\bar C_5$ has been defined in Lemma \ref{control on H/R}. It results
\begin{align*}
\pa_\nu u(x) &=  \left\langle \nabla u(x) - \gamma \frac{\sqrt{H(\tilde x, \tilde R)}}{\tilde R}e_N , \nu \right\rangle + \gamma \frac{\sqrt{H(\tilde x, \tilde R)}}{\tilde R}\langle e_N,\nu \rangle \\
& \ge - \eps(\nu) + \gamma \bar C_5 \langle e_N,\nu \rangle >0
\end{align*}
for every $x \in \{x_N>M_\nu\}$, where $M_\nu := M_{\eps(\nu)}$ has been defined in Lemma \ref{C^1 approx}. The same argument gives the monotonicity of $v$ for $x_N \ll 1$.
\end{proof}

With a slightly modification of the conclusion of the proof, we obtain also the

\begin{corollary}\label{uniform monotonicity at infinity}
If we consider $\Theta:=\{ \nu \in \S^{N-1}: \langle e_N, \nu \rangle \ge \hat C \}$ with $\hat C \in (0,1]$, then there exists $M_{\Theta}>0$ such that  
\begin{align*}
& x \in\{x_N > M_{\Theta}\} \Rightarrow \pa_\nu u(x) >0 \qquad \forall \nu \in \Theta \\
& x \in\{x_N < -M_{\Theta}\} \Rightarrow \pa_\nu v(x) <0 \qquad \forall \nu \in \Theta.
\end{align*}
\end{corollary}

\section{Monotonicity in the $e_N$ direction}\label{sec: monotonicity}

We are going to apply the moving planes method in order to show that $u$ and $v$ are monotone in the $e_N$ direction in the whole $\R^N$.  To be precise:

\begin{proposition}\label{monotonicity}
Let $(u,v)$ be a solution of \eqref{system} satisfying \eqref{alg growth} and \eqref{uniform limit}. Then
\[
\frac{\pa u}{\pa x_N}>0 \quad \text{and} \quad \frac{\pa v}{\pa x_N}<0 \quad \text{in $\R^N$}.
\]
\end{proposition}

In what follows we will use many times the following version of the maximum principle in unbounded domains, Lemma 2.1 in \cite{BeCaNi}.

\begin{lemma}\label{max_principle unbounded}
Let $D$ be an open connected subset of $\R^N$, possibly unbounded. Assume that $\overline{D}$ is disjoint from the closure of an infinite open connected cone. Suppose that, for a function $c \in L^{\infty}_{loc}(D)$, $c \le 0$ a.e. in $D$, we have
\[
\begin{cases}
\Delta v + c(x) v \ge 0 & \text{in $D$} \\ 
v \le 0 & \text{on $\pa D$},
\end{cases}
\]
where $v \in \mathcal{C}^0(\overline{D}) \cap W_{loc}^{2,N}(D)$ and $v^+ \in L^\infty(D)$, that is, $v$ is bounded above. Then $v \le 0$ in $D$.
\end{lemma}

We postpone the proof of Proposition \ref{monotonicity} after the following Lemma, which is a consequence of the uniform estimate given in Corollary \ref{corol 4.10 in Wa}.

\begin{lemma}\label{bound on strips}
Let $(u,v)$ be a solution of \eqref{system} satisfying \eqref{alg growth} and \eqref{uniform limit}. Then for every $M>0$ there exists $\bar C_M>0$ such that
\begin{align*}
& u(x) + |\nabla u(x)| \le \bar C_M \quad \forall x \in \R^{N-1} \times (-\infty,M], \\
& v(x) + |\nabla v(x)| \le \bar C_M \quad \forall x \in \R^{N-1} \times [-M,+\infty).
\end{align*}
\end{lemma}
\begin{proof}
We prove only the first inequality. Under our assumptions, we know that $(u,v)$ has linear growth (see Proposition \ref{uniqueness asymptotic profile thm}). For any $x \in \R^N$, let $\tilde x \in \{|u-v|<\bar C_3\}$ such that $\tilde x'=x'$ and let $\tilde R= \frac{3}{2}|x_N- \tilde x_N|$, so that $x \in B_{\tilde R}(\tilde x)$ ($\tilde x$ exists, see Lemma \ref{about u=v}). By means of Corollary \ref{corol 4.10 in Wa} we deduce that
\begin{equation}\label{eq24}
u(x)+v(x) \le \sup_{y \in B_{\tilde R}(\tilde x)} \bar C_4 \left(1+ \frac{3}{2}|x_N-\tilde x_N|\right) \le \frac{3}{2}\bar C_4 \left( \frac{2}{3}+ \zeta+|x_N|\right) \qquad \forall x \in \R^N,
\end{equation}
where $\zeta$ has been defined in Remark \ref{def zeta}. Now, let $M_1$ be defined in Remark \ref{rem su corol 1}, so that $u v^2 \le C e^{-C \vert x_N \vert}$ in $\{x_N<-M_1\}$. Moreover, by \eqref{uniform limit} there exist $M_3>0$ such that $u\le 1$ in $\R^{N-1} \times (-\infty,-M_3+ \frac{1}{2}]$. we set $M_4:= \max\{M_1,M_3\}$ and we take any $M>M_4$.

\medskip

\noindent By \eqref{eq24}, it results
\[
u(x',x_N) \le  \begin{cases} \frac{3}{2}\bar C_4 \left( \frac{2}{3}+ \zeta+M\right) & \text{if $x \in \{|x_N| \le M\}$} \\
                                          1 & \text{if $x \in \{x_N \le -M\}$} \end{cases} \le 1+\frac{3}{2}\bar C_4 \left( \frac{2}{3}+ \zeta+|x_N|\right)=: C_{1,M}
\]
whenever $(x',x_N) \in \R^{N-1} \times (-\infty,M]$. Clearly, if $M \le M_4$ the same bound holds. 

\medskip

\noindent Let's pass to the estimate on the gradient. In $\R^{N-1} \times \left[-M-\frac{1}{2},M+\frac{1}{2}\right]$ both $u$ and $u v^2$ are uniformly bounded thanks to \eqref{eq24}. Also, by definition of $M_1$ and $M_3$ both $u$ and $u v^2$ are uniformly bounded in $\R^{N-1} \times (-\infty,-M]$. Altogether, this means that $u$ and $u v^2$ are uniformly bounded in $\R^{N-1} \times \left(-\infty,M + \frac{1}{2}\right]$, so that we can apply the standard gradient estimates for the Poisson equation (see \cite{GiTr}, section 3.4) in cubes of side $1$, obtaining the existence of $C_{2,M}>0$ such that $|\nabla u(x)| \le C_{2,M}$ for every $x \in \R^{N-1} \times (-\infty,M]$.

\medskip

\noindent The thesis is then satisfied with $\bar C_M:= \max\{C_{1,M}, C_{2,M}\}$.
\end{proof}

\begin{proof}[Proof of Proposition \ref{monotonicity}]
We introduce the classical notation for the moving planes method: for $\lambda \in \R$, we set 
\[
u_{\lambda}(x',x_N):=u(x',2\lambda - x_N) \quad \text{and} \quad T_\l:= \{x_N > \l\}.
\]
We aim at proving that 
\begin{equation}\label{moving thesis 0}
u_\l(x) \le u(x) \quad \text{and} \quad v_\l(x) \ge v(x) \quad \forall x \in T_\l, \ \forall \l \in \R,
\end{equation}
This and the strong maximum principle give the desired monotonicity.\\
To prove that \eqref{moving thesis 0} is satisfied, we show that
\begin{equation}\label{moving thesis}
\Sigma:=\{ \l \in \R: \text{$u_\t \le u$ and $v_\t \ge v$ in $T_\t$ for every $\t \ge \l$}\}= \R.
\end{equation}

\paragraph{Step 1)} \emph{There exists $\bar M>0$ such that if $\l>\bar M$ then $u_\l \le  u$ and $v_\l \ge v$ in $T_\l$}.\\
Let $M_N:= M_{e_N}$, where $M_{e_N}$ has been defined in Proposition \ref{monotonicity at infinity}. Let $K:= \sup\{u: x_N<M_{N}\} <+\infty$. By assumption \eqref{uniform limit}, for every $\d>0$ there exists $\bar M>0$ such that 
\begin{equation}\label{eq7}
u(x)>K \quad \text{and}  \quad v(x)<\d \quad \text{in $\{x_N> 2\bar M-M_N\}$}.
\end{equation}
Let $\l>\bar M$. If $x \in \{x_N\ge 2\lambda -M_{N}\}$ then $x_N \ge 2 \bar M-M_N$ and $2\lambda - x_N \le M_N$, so that by definition
\[
u_\l(x) = u(x', 2\l - x_N) \le K \le u(x).
\]
To prove that $u_\l \le u$ in $T_\l$ for every $\l > \bar M$, it remains to show that if $\l> \bar M$ then $u_\l \le u$ in $\{\l<x_N<2\l-M_N\}$. If $x \in \{\l<x_N<2\l- M_N\}$, then $x_N>2\l-x_N>M_N$, so that the fact that $u_\l(x) \le u(x)$ follows directly from the monotonicity of $u$ in the $e_N$ direction for $\{x_N>M_N\}$.   \\
Now, let us show that if $\l>\bar M$ then $v_\l \ge v$ in $T_\l$. Since $u_\l \le u$ in $T_\l$, we have
\[
\begin{cases}
\Delta(v-v_\lambda)-u_\lambda^2(v-v_\l)\ge 0 & \text{in $T_\l$}\\
v-v_\l =0 & \text{on $\pa T_\l$},
\end{cases}
\]
and $(v-v_\l)^+ \le v \le \delta$ in $T_\l$ (see equation \eqref{eq7}). Consequently, we are in position to apply Lemma \ref{max_principle unbounded}, obtaining $v-v_\l \le 0$ in $T_\l$.
\paragraph{Step 2)} \emph{$\Sigma = \R$}. \\
In the first step we showed that $\Sigma \neq \emptyset$. Note that $\Sigma$ is a closed interval and contains the unbounded interval $(\bar M,+\infty)$. Assume by contradiction that $\Sigma \neq \R$, that is, $\Lambda := \inf \Sigma > -\infty$. Then there exist sequences $(\lambda_i) \subset \R$ and $(x^i) \subset T_{\l_i}$ such that
$\lambda_i < \Lambda$ and $\l_i \to \Lambda$ as $i \to \infty$, and at least one between
\begin{subequations}
\begin{align}
& u_{\l_i}(x^i) > u(x^i) \quad \forall i \label{eq8}\\ 
& v_{\l_i}(x^i)<v(x^i) \quad \forall i, \label{eq13}
\end{align}
\end{subequations}
holds true.  

\medskip

\noindent Assume that \eqref{eq8} holds true. We claim that the sequence $(x_N^i) \subset \R$ is bounded. If not, as $x_N^i>\l_i$ and $\l_i$ is bounded, up to a subsequence $x_N^i \to +\infty$ as $i \to \infty$. It follows that $2 \l_i -x_N^i \to -\infty$, and in light of assumption \eqref{uniform limit} we obtain
\[
\lim_{i \to \infty} u_{\l_i}(x^i) = \lim_{i \to \infty} u((x^i)',2\l_i -x_N^i)  = 0 \quad \text{and} \quad \lim_{i \to \infty} u(x^i) = +\infty,
\]
in contradiction with \eqref{eq8} for $i$ sufficiently large. Hence the claim is proved and, up to a subsequence, $x_N^i \to x_N^\infty$ as $i \to \infty$. 

\medskip

\noindent Let us set
\[
u^i(x):= u((x^i)'+x',x_N) \quad \text{and} \quad v^i(x):= v((x^i)'+x',x_N).
\]
From Lemma \ref{bound on strips} it follows that $\{(u^i,v^i)\}$ is uniformly bounded and equi-Lipschitz-continuous in any compact subset of $\R^N$, so that the standard regularity theory for elliptic equations (see again \cite{GiTr}) implies that up to a subsequence $(u^i,v^i)$ converges in $\mathcal{C}^2_{loc}(\R^N)$ to a pair $(u^\infty, v^\infty)$, still solution of \eqref{system} in $\R^N$. 

\medskip

\noindent We wish to show that $x_N^\infty=\Lambda$. From the absurd assumption, equation \eqref{eq8}, we get
\begin{equation}\label{eq9}
\begin{split}
u_\Lambda^{\infty}(0',x_N^\infty) &= u^\infty(0',2\Lambda - x_N^\infty) = \lim_{i \to \infty} u( (x^i)', 2\l_i-x_N^i) \\
&= \lim_{i \to \infty} u_{\l_i}(x^i) \ge \lim_{i \to \infty} u(x^i) = u^\infty (0',x_N^{\infty}).
\end{split}
\end{equation}
Let us observe that $((x^i)'+x',x_N) \in T_\Lambda$ whenever $(x',x_N) \in T_\Lambda$. By definition of $\Lambda$, $u_\Lambda \le u$ in $T_\Lambda$.
Consequently, by the convergence of $u^i$ to $u^\infty$ we deduce
\begin{align*}
u^{\infty}_\Lambda(x',x_N) &= \lim_{i \to \infty} u^i(x',2\Lambda - x_N) = \lim_{i \to \infty} u((x^i)'+x',2 \Lambda- x_N) \\
& \le \lim_{i \to \infty} u((x^i)'+x',x_N) = \lim_{i \to \infty} u^i(x',x_N) = u^\infty(x',x_N)
\end{align*}
for every $(x',x_N) \in T_\Lambda$. Analogously, as $v_\Lambda \ge v$ in $T_\Lambda$, we have $v^\infty_\Lambda \ge v^\infty$ in $T_\Lambda$.\\ 
Now,
\begin{equation}\label{eq12}
\begin{cases}
-\Delta (u^\infty - u^\infty_\Lambda) + (v^\infty)^2 (u^\infty- u^\infty_\Lambda)= ((v^\infty_\Lambda)^2 - (v^\infty)^2) u^\infty_\Lambda \ge 0 & \text{in $T_\Lambda$} \\
u^\infty-u^\infty_\Lambda \ge 0 & \text{in $T_\Lambda$} \\
u^\infty -u^\infty_\Lambda = 0 & \text{on $\pa T_\lambda$}.
\end{cases}
\end{equation}
Furthermore, $u^\infty-u_\Lambda^\infty$ is not identically $0$: indeed by assumption \eqref{uniform limit}
\[
\lim_{x_N \to +\infty}   u^\infty(x',x_N) - u^\infty_\Lambda(x',x_N) = +\infty.
\]
Hence, the strong maximum principle implies that necessarily $u^\infty- u^\infty_\Lambda > 0$ in $T_\Lambda$. A comparison with \eqref{eq9} reveals that
\begin{equation}\label{eq10}
x_N^\infty=\Lambda.
\end{equation}

\medskip

\noindent Now, by the absurd assumption \eqref{eq8} we deduce
\[
0< u_{\l_i}(x^i) - u(x^i) = u^i(x', 2\l_i -x^i_N)-u^i(x',x_N) = 2\pa_N u^i (x', \xi^i)(\l_i-x^i_N) \qquad \forall i;
\]
As $\l_i < x_N^i$ for every $i$ this implies $\pa_N u^i(x', \xi^i_N)< 0$ for every $i$. As $\l_i \to \Lambda$ and $x_N^i \to \Lambda$ as $i \to \infty$, passing to the limit as $i \to \infty$ we deduce
\begin{equation}\label{eq11}
\pa_N u^\infty(0', \Lambda) \le 0.
\end{equation}
On the other hand, thanks to the \eqref{eq12} and the fact that $u^\infty-u^\infty_\Lambda>0$ in $T_\Lambda$, we are in position to apply the Hopf' Lemma:
\[
\pa_\nu (u^\infty(0',\Lambda) -u^\infty_\Lambda(0',\Lambda))<0,
\]
which means
\[
2\pa_N u^\infty(0',\Lambda)>0,
\]
in contradiction with \eqref{eq11}. 

\medskip

The above argument says that \eqref{eq8} cannot occur. With minor changes, we can show that also \eqref{eq13} is not verified, so that $\Sigma=\R$, which completes the proof.
\end{proof}

\section{$1$-dimensional symmetry}\label{sec:symmetry}

In this section we complete the proof of our main result, Theorem \ref{main thm}. We will follow the technique introduced by the first author in \cite{Fa1}: we will show that, starting from Proposition \ref{monotonicity}, it is possible to prove that $\pa_\nu u >0$ and $\pa_\nu v<0$ for every $\nu \in \S^{N-1}_+= \{\nu \in \S^{N-1}: \nu_N>0 \}$. The conclusion will follow easily.

\begin{proposition}
Let $(u,v)$ be a solution of \eqref{system} satisfying \eqref{alg growth} and \eqref{uniform limit}. Then $(u,v)$ depends only on $x_N$.
\end{proposition}

\begin{proof} We divide the proof in several steps.
\paragraph{Step 1)} \emph{For every $\sigma>0$ there exists $\eps=\eps(\sigma)>0$ such that 
\[
\pa_N u(x) \ge \eps \quad \text{and} \quad \pa_N v(x) \le -\eps \quad \forall x \in \overline{S}_\sigma,
\]
where $S_\s:= \R^{N-1} \times (-\sigma,\sigma)$}. \\
By contradiction, fixed $\sigma>0$, assume that there exists $(x^i) \subset S_\s$ such that at least one between
\begin{subequations}
\begin{align}
& \lim_{i \to +\infty} \frac{\pa u	}{\pa x_N}(x^i)= 0 \label{eq14} \\
& \lim_{i \to +\infty} \frac{\pa u	}{\pa x_N}(x^i)= 0 \label{eq15}
\end{align}
\end{subequations}
holds true. Only to fix our minds, assume that \eqref{eq14} holds. We define
\[
u^i(x):= u(x+x^i) \quad \text{and} \quad v^i(x):= v(x+x^i).
\]
Note that $|x_N^i| \le \sigma$ for every $i$, so that for any compact set $K \subset \R^N$ there exists $M>0$ such that $x+x^i \in S_{M}$ for every $x \in K$. Lemma \ref{bound on strips} and standard elliptic estimates say that, up to a subsequence, $(u^i,v^i) \to (u^\infty,v^\infty)$ in $\mathcal{C}^2_{loc}(\R^N)$, where $(u^\infty, v^\infty)$ is still a solution to \eqref{system}. By the convergence, we have
\[
\frac{\pa u^\infty}{\pa x_N} \ge 0   \quad \text{and} \quad \frac{\pa v^\infty}{\pa x_N} \le 0 \quad \text{in $\R^N$},
\]
and $\pa_N u^\infty(0)=0$. Furthermore,
\[
-\Delta \left(\pa_N u^\infty \right) +(v^\infty)^2 \left(\pa_N u^\infty \right) = -2 u^\infty v^\infty \left(\pa_N v^\infty \right) \ge 0 \quad \text{in $\R^N$}.
\]
The strong maximum principle implies that either $\pa_N u^\infty>0$ or $\pa_N u^\infty \equiv 0$. The former one is in contradiction with the fact that $\pa_N u^\infty(0)=0$, the latter one is in contradiction with assumption \eqref{uniform limit}, which is also satisfied by the limiting profile $(u^\infty,v^\infty)$. Thus, \eqref{eq14} cannot occur. A similar argument shows that also \eqref{eq15} does not hold. 

\paragraph{Step 2)} \emph{For every $\sigma>0$, the map $\nu \mapsto (\pa_\nu u, \pa_\nu v)$ is in $\mathcal{C}^{0,1} \left(\S^{N-1}, \left(\mathcal{C}^0(\overline{S}_\s)\right)^2\right)$}.\\
By Lemma \ref{bound on strips}, we know that $|\nabla u|+ |\nabla v|\le \bar C_\s$ in $\overline{S}_\s$. Hence
\[
\left|\frac{\pa u}{\pa \nu_1}(x) -\frac{\pa u}{\pa \nu_2}(x) \right| + \left|\frac{\pa v}{\pa \nu_1}(x) -\frac{\pa v}{\pa \nu_2}(x) \right| \le 2 \bar C_\s |\nu_1-\nu_2|
\]
for every $x \in \overline{S}_\s$.

\paragraph{ Step 3)} \emph{$u$ is strictly increasing and $v$ is strictly decreasing with respect to all the unit vectors of an open neighborhood of $e_N$ in $\S^{N-1}$}.\\
Let $\Theta:=\left\{\nu \in \S^{N-1}: \langle e_N,\nu \rangle \ge \frac{1}{2} \right\}$. By Corollary \ref{uniform monotonicity at infinity}, we know that there exists $M_{\Theta}$ such that 
\[
\frac{\pa u}{\pa \nu} > 0 \quad \text{in $\{x_N>M_\Theta\}$} \quad\text{and} \quad \frac{\pa v}{\pa \nu}<0 \quad \text{in $\{x_N<-M_\Theta\}$}, 
\]
for every $\nu \in \Theta$. Let $\s > M_\Theta$. Using steps 1) and 2), we deduce that there exists an open neighborhood $\mathcal{O}_{e_N}$ of $e_N$ in $\S^{N-1}$ such that
\begin{equation}\label{eq16}
\frac{\pa u}{\pa \nu}  (x) > 0 \quad \text{and} \quad \frac{\pa v}{\pa \nu}(x) <0 \quad \forall x \in S_\s, \ \forall \nu \in \mathcal{O}_{e_N}.
\end{equation}
We can assume that $\mathcal{O}_{e_N} \subset \Theta$ (if not, we replace $\mathcal{O}_{e_N}$ with a smaller neighborhood). This means that, for every $\nu \in \mathcal{O}_{e_N}$, it results
\[
\frac{\pa u}{\pa \nu} > 0 \quad \text{in $\{x_N>-\sigma\}$} \quad\text{and} \quad \frac{\pa v}{\pa \nu}<0 \quad \text{in $\{x_N<\sigma\}$}, 
\]
Furthermore, for every $\nu \in \mathcal{O}_{e_N}$
\[
\begin{cases}
\Delta (-\pa_\nu u) -v^2 (-\pa_\nu u) = -2 uv \pa_\nu v \ge 0 & \text{in $\R^{N-1} \times (-\infty,-\sigma)$}\\
-u_\nu \le 0 & \text{on $\pa \left(\R^{N-1} \times (-\infty,-\sigma)\right)$} \\
-\pa_\nu u \in L^\infty\left(\R^{N-1} \times (-\infty,-\sigma)\right),
\end{cases}
\]
where the last one follows from Lemma \ref{bound on strips}. We are then in position to apply Lemma \ref{max_principle unbounded}, obtaining $\pa_\nu u \ge 0$ in $\R^{N-1} \times (-\infty,-\sigma)$. Together with \eqref{eq16}, this gives $\pa_\nu u \ge 0$ in $\R^N$ for every $\nu \in \mathcal{O}_{e_N}$. Similarly, from
\[
\begin{cases}
\Delta (\pa_\nu v) -u^2 (\pa_\nu v) = 2 uv \pa_\nu u \ge 0 & \text{in $\R^{N-1} \times (\sigma,+\infty)$}\\
v_\nu \le 0 & \text{on $\pa \left(\R^{N-1} \times (\sigma,+\infty)\right)$} \\
\pa_\nu v \in L^\infty\left(\R^{N-1} \times (\sigma,+\infty)\right),
\end{cases}
\]
we deduce $\pa_\nu v \le 0$ in $\R^N$ for every $\nu \in \mathcal{O}_{e_N}$. Finally, the strong maximum principle provides $\pa_\nu u>0$ and $\pa_\nu v <0$ in $\R^N$, for every $\nu \in \mathcal{O}_{e_N}$.   

\paragraph{Step 4)} \emph{$u$ is strictly increasing and $v$ is strictly decreasing with respect to all the directions of the upper hemisphere $\S^{N-1}_+= \{ \nu \in \S^{N-1}: \langle e_N,\nu \rangle >0\}$}. \\
Let $\Omega$ be the set of $\nu \in \S^{N-1}_+$ for which there exists an open neighborhood $\mathcal{O}_\nu \subset \S^{N-1}$ of $\nu$ such that 
\[
\frac{\pa u}{\pa \mu}>0 \quad \text{and} \quad \frac{\pa v}{\pa \mu}<0 \quad \text{in $\R^N$}, \ \forall \mu \in \mathcal{O}_\nu.
\]
The set $\Omega$ is open by definition, and contains $e_N$ for the previous step. If we show that it is closed with respect to the topology of $\S^{N-1}_+$, then $\Omega= \S^{N-1}_+$ and the claim is proved. Let $\bar \nu$ be a cluster point of $\Omega$ (note that $\langle e_N, \bar \nu\rangle >0$), that is, there exists $(\nu_n) \subset \Omega$ such that $\nu_n \to \bar \nu$. As
\[
\frac{\pa u}{\pa \nu_n}>0 \quad \text{and} \quad \frac{\pa v}{\pa \nu_n}<0 \quad \text{in $\R^N$}, \ \forall n,
\]
by continuity
 \[
\frac{\pa u}{\pa \bar \nu}\ge 0 \quad \text{and} \quad \frac{\pa v}{\pa \bar \nu}\le 0 \quad \text{in $\R^N$}.
\]
The strong maximum principle implies that or $\pa_{\bar \nu} u \equiv 0$ or $\pa_{\bar \nu} u>0$ in $\R^N$; analogously, $\pa_{\bar \nu} v \equiv 0$ or $ \pa_{\bar \nu} v<0$ in $\R^N$. As $\bar \nu$ is not orthogonal to $e_N$, assumption \eqref{uniform limit} says that neither $\pa_{\bar \nu} u\equiv 0$ nor $\pa_{\bar \nu} v \equiv 0$ can be satisfied, thus $\pa_{\bar \nu} u>0$ and $\pa_{\bar \nu} v<0$ in $\R^N$. It remains to show that there exists an open neighborhood $\mathcal{O}_{\bar \nu}$ of $\bar \nu$ in $\mathcal{S}^{N-1}_+$ such that for every $\mu \in \mathcal{O}_{\bar \nu}$ 
\[
\frac{\pa u}{\pa \mu}>0 \quad \text{and} \quad \frac{\pa v}{\pa \mu}<0 \quad \text{in $\R^N$}.
\]
It is possible to adapt the same proof of steps 1) to 3) with minor changes, in order to deduce the existence of $\mathcal{O}_{\bar \nu}$ (in the third step we replace $\Theta$ with $\{ \nu \in \mathbb{S}^{N-1}: \langle e_N,\nu\rangle \ge \frac{1}{2} \langle e_N, \bar \nu > 0\rangle\}$). Consequently, $\bar \nu \in \Omega$ and $\Omega$ is closed with respect to the topology of $\S^{N-1}_+$.

\paragraph{Step 5)} \emph{Conclusion of the proof}. \\
Since $\Omega= \S^{N-1}_+$, by continuity we have
\[
\frac{\pa u}{\pa \nu} \ge 0 \quad \text{and} \quad \frac{\pa v}{\pa \nu} \le 0 \quad \text{in $\R^N$}
\]
for every $\nu$ which is orthogonal to $e_N$. But also $-\nu$ is orthogonal to $e_N$, so that
\[
\frac{\pa u}{\pa \nu} \equiv 0 \quad \text{and} \quad \frac{\pa v}{\pa \nu} \equiv 0 \quad \text{in $\R^N$}
\]
for every $\nu$ orthogonal to $e_N$. In particular 
\[
\frac{\pa u}{\pa x_i} \equiv 0 \quad \text{and} \quad \frac{\pa v}{\pa x_i} \equiv 0 \quad \text{in $\R^N$},  \text{for $i=1,\ldots,N-1$}. \qedhere
\]
\end{proof}
 
\section{Proof of Corollary \ref{main corol}}\label{sec:corollary}

We will show that if $(u,v)$ is a solution of \eqref{system} with algebraic growth and \eqref{limit at infinity} holds true, then \eqref{uniform limit} is satisfied.

\begin{proof}[Proof of Corollary \ref{main corol}]
Firstly, let us observe that, since $u,v>0$, \eqref{limit at infinity} implies
\begin{equation}\label{limit splitted}
\lim_{x_N \to +\infty} u(x',x_N) = +\infty \quad \text{and} \quad \lim_{x_N \to -\infty} v(x',x_N)= +\infty
\end{equation}
uniformly in $x' \in \R^{N-1}$. Thus, in order to obtain the thesis it remains to show that under \eqref{alg growth} and \eqref{limit at infinity} we have
\begin{equation}\label{th corol}
\lim_{x_N \to -\infty} u(x',x_N) = 0 \quad \text{and} \quad \lim_{x_N \to +\infty} v(x',x_N) = 0
\end{equation}
We prove only the second one in \eqref{th corol}, for the first one it is possible to argue in the same way.

\paragraph{Step 1)} \emph{under \eqref{alg growth} and \eqref{limit at infinity}, $(u,v)$ has linear growth}.\\
Given $K>0$, by \eqref{limit at infinity} there exists $M>0$ such that $u(x)>K$ if $x \in \{x_N>M/2\}$. For an arbitrary $\theta>1$, if $x \in \left\{x_N> M, \ |x'| < \theta x_N\right\}$ the ball $B_x := B_{x_{N}/100}(x)$ is contained in $\left\{x_N>M/2, \ |x'|< 2\theta x_N \right\}$. Consequently, if $x \in \left\{x_N>M, \ |x'| < \theta x_N\right\}$ we have 
\[
u(y) \ge K_x:= \inf_{z \in B_x} u(z) \ge K \qquad \forall y \in B_x, 
\]
and
\[
v(y) \le C (1+|y|^p) \le C\left(1+ (2 \theta +1)^p y_N^p) \right) \le C(1+ x_N^p) \qquad \forall y \in B_x.
\]
The latter one gives $\delta_x:=\sup_{y \in B_x} v(y) \le C (1+x_N^p)$. Now,
\[
\begin{cases}
-\Delta v \le -K^2 v & \text{in $B_x$} \\
v \ge 0 & \text{in $B_x$}  \\
v \le \delta_x & \text{in $B_x$},
\end{cases}
\]
and we are in position to apply Lemma \ref{lemma exponential decay}: it follows 
\begin{equation}\label{eq17}
v(x) \le C \delta_x e^{-C K x_N} \le C (1+x_N^p) e^{-C K x_N} \qquad \forall x \in \left\{x_N>M, \ |x'| < \theta x_N\right\}.
\end{equation}
Let us consider the blow-down family $(u_{0,R},v_{0,R})=:(u_R,v_R)$. In light of the algebraic growth of $(u,v)$, Theorem \ref{thm blow down} applies: there exists a homogeneous harmonic polynomial $\Psi$ of degree $d \in \N \setminus \{0\}$ such that, up to a subsequence, $(u_R,v_R)$ converges to $(\Psi^+,\Psi^-)$ in $\mathcal{C}^0_{loc}(\R^N)$ as $R \to +\infty$. On the other hand, let $x \in \left\{|x'| < \theta x_N\right\}$; there exists $R_x>0$ such that $Rx \in  \left\{x_N>M, \ |x'| < \theta \pi x_N\right\}$ for every $R>R_x$. By means of \eqref{eq17}, we deduce that
\[
\lim_{R \to +\infty} v_R(x) = \lim_{R \to +\infty} \frac{1}{\sqrt{H(0,R)}} v(Rx)  = 0 \qquad \forall x \in \left\{|x'| < \theta x_N\right\},
\]
where we used also Corollary \ref{cor blow-down} to ensure that $H(0,R)$ does not tend to $0$. As $\theta$ has been arbitrarily chosen, we deduce that $v_R \to 0$ pointwise in $\R^N_+$. By the uniqueness of the limit, $\Psi$ has to be a homogeneous harmonic polynomial which vanishes in the entire half-space $\R^N_+$: as showed in the proof of Proposition \ref{uniqueness asymptotic profile thm}, necessarily $\Psi$ is a linear function and $d=1$. By means of Corollary \ref{from N to growth}, we deduce that $(u,v)$ has linear growth.
\paragraph{Step 2)} \emph{Conclusion of the proof}.\\
As $(u,v)$ has linear growth, we can choose $\bar C_3$ as in Remark \ref{def di bar C_3}. Assumption \eqref{limit at infinity} it is sufficient to ensure that the geometry of the set $\{|u-v|< \bar C_3\}$ is described by Lemma \ref{about u=v}: $\{|u-v|< \bar C_3\}$ is bounded in the $e_N$ direction and unbounded in all the other directions. Consequently, also Lemma \ref{union of the cones} applies: for $\hat R \ge \bar R$ we can find $M_2$ as in the quoted statement. \\
Given $K>0$, by \eqref{limit splitted} there exists $M>0$ such that if $x \in \left\{x_N> \frac{M}{2}\right\}$ then $u(x) \ge K$. Let $M_5:= \max\left\{M,M_2\right\}$, so that
\[
\{x_N>M_5\} \subset \bigcup_{\substack{x_0 \in \{|u-v|< \bar C_3\} \\
R > \hat R  }   }  S_{x_0,R}^+.
\]
If $x \in \left\{x_N > M_5\right\}$ then the ball $B_x:= B_{x_N/100}(x)$ is contained in $\left\{x_N>\frac{M}{2} \right\}$, so that
\[
\begin{cases}
-\Delta v \le -K^2 v & \text{in $B_x$} \\
v \ge 0 & \text{in $B_x$}  \\
v \le \delta_x & \text{in $B_x$},
\end{cases}
\]
where $\delta_x := \sup_{B_x} v<+\infty$, because $v \in L^\infty_{loc}(\R^N)$. From Lemma \ref{lemma exponential decay} we obtain
\begin{equation}\label{stima da migliorare}
v(x) \le C \left(\sup_{y \in B_x} v(y)\right) e^{-C K x_N}.
\end{equation}
To control $\sup_{B_x} v$, we consider $\tilde x$ and $\tilde R$ defined in Lemma \ref{union of the cones} and Remark \ref{rem sui quadrati}. As $B_x \subset Q_x$, a fortiori $B_x \subset S_{\tilde x, \tilde R}^+ \subset B_{\tilde R}(\tilde x)$. We are then in position to apply Corollary \ref{corol 4.10 in Wa}: 
\begin{align*}
\sup_{y \in B_x} v(y)  &\le \bar C_4(1+\tilde R)  = \bar C_4\left(1+ \frac{3}{2}(x_N-\tilde x_N)\right) \\
& \le \bar C_4\left(1+ \frac{3}{2}\zeta+ \frac{3}{2}x_N\right) \le C x_N
\end{align*}
provided $x_N$ is sufficiently large (recall the definition of $\zeta$, Remark \ref{def zeta}). Plugging into \eqref{stima da migliorare}, we see that for every $x$ such that $x_N \gg 1$ is sufficiently large it results 
\[
v(x) \le C x_N e^{-C K x_N},
\]
which gives the second limit in \eqref{th corol}.
\end{proof}

\appendix 
 
\section{Appendix}\label{sec:appendix}

For the reader's convenience, we report some known and few new results which we used many times in our work. We prefer to write down explicitly the statements below, because in the literature they do not appear always in this form, and because sometimes the proofs are missing. In such a case, we will write them for the sake of completeness.

\subsection*{The exponential decay}
 
It is by now well known that, if $(u,v)$ solves \eqref{system} and $u$ is very large in a ball $B_{2r}(x_0)$, then $v$ has to be exponentially small with respect to $u$ in a smaller ball.

\begin{lemma}[Lemma 4.4 in \cite{CoTeVe}]\label{lemma exponential decay}
Let $x_0 \in \R^N$ and $r>0$. Let $u \in H^1(B_{2r}(x_0))$ be such that
\[
\begin{cases}
-\Delta v \le -K v & \text{in $B_{2r}(x_0)$} \\
v \ge 0 & \text{in $B_{2r}(x_0)$} \\
v \le A & \text{on $\pa B_{2r}(x_0)$},
\end{cases}
\]
where $K$ and $A$ are two positive constants. Then for every $\alpha \in (0,1)$ there exists $C_\alpha>0$, not depending on $A$, $K$, $R$ and $x_0$, such that
\[
\sup_{x \in B_r(x_0)} v(x) \le \alpha A  e^{-C_\alpha K^{1/2} r}.
\]
\end{lemma}  

We will always apply this result with $\alpha=1/2$ to simplify the notation.

\subsection*{The segregation theorem}

Let us consider the problem 
\begin{equation}\label{system beta}
\begin{cases}
-\Delta u_\beta= - \beta u_\beta v_\beta^2  \\
-\Delta v_\beta=- \beta u_\beta^2 v_\beta \\
u_\beta>0, v_\beta>0 ,
\end{cases}
\end{equation}
 where $\beta$ is a positive parameter tending to $+\infty$. The following is the local version of the uniform H\"older estimates obtained in \cite{NoTaTeVe}, which has been proved in \cite{Wa}.
 
\begin{theorem}\label{local Holder estimate}
Let $\{(u_\beta,v_\beta)\}$ be a family of solutions to \eqref{system beta} in a ball $B_{2r}(x_0) \subset \R^N$ (where $x_0 \in \R^N$ and $r>0$). Assume that, as $\beta \to +\infty$, $\{ (u_\beta,v_\beta)\}$ is uniformly bounded in $L^\infty(B_{2r}(x_0))$. Then $\{ (u_\beta,v_\beta)\}$ is uniformly bounded in $\mathcal{C}^{0,\alpha}(B_r(x_0))$, for every $\alpha \in (0,1)$.
\end{theorem}

As a consequence, one can easily adapt the proof of Theorem 1.2 of \cite{NoTaTeVe} and obtain a local segregation theorem, see also \cite{DaWaZh, TaTe}.

\begin{theorem}\label{segregation thm}
Let $\{(u_\beta,v_\beta)\}$ be a family of solutions to \eqref{system beta} in a ball $B_{2r}(x_0) \subset \R^N$ (where $x_0 \in \R^N$ and $r>0$). Assume that, as $\beta \to +\infty$, $\{ (u_\beta,v_\beta)\}$ is uniformly bounded in $L^\infty(B_{2r}(x_0))$. Then there exists a pair $(u_\infty,v_\infty)$ such that, up to a subsequence, there holds
\begin{itemize}
\item[($i$)] $u_\beta \to u_\infty$ and $v_\beta \to v_\infty$ in $\mathcal{C}^0(B_r(x_0)) \cap H^1(B_r(x_0))$,
\item[($ii$)] $u_\infty v_\infty \equiv 0$ in $B_r(x_0)$ and 
\[
\lim_{\beta \to +\infty} \int_{B_r(x_0)} \beta u^2_\beta v_\beta^2 = 0 ,
\]
\item[($iii$)] the limiting profile satisfies 
\[
\begin{cases} 
-\Delta u_\infty=0 & \text{in $\{u_\infty>0\} \cap B_r(x_0)$} \\
-\Delta v_\infty=0 & \text{in $\{v_\infty>0\} \cap B_r(x_0)$},
\end{cases}
\]
\item[($iv$)] $u_\infty-v_\infty$ is harmonic and both $u_\infty$ and $v_\infty$ are subharmonic in $B_r(x_0)$.
\end{itemize}
\end{theorem}

\begin{remark}\label{rem su CPAM e dim}
In \cite{NoTaTeVe} it is considered a different system with some additional terms. In particular, the term $u^3$ appear in the equation for $u$, and $v^3$ in the equation for $v$. Since it is required that these powers are subcritical for the Sobolev embedding, this imposes a restriction on the dimension $N$. However, as explained in the introduction of the quoted paper, all the results are valid in any dimension provided $u^3$ and $v^3$ are replaced by subcritical terms; this is clearly the case of system \eqref{system beta}.
\end{remark}

\subsection*{The Almgren monotonicity formula}

We recall some properties of the functions $H$ and $N$, defined in \eqref{def di H,N}.
Firstly
\begin{remark}
A direct computation shows that
\[
\frac{\pa}{\pa r}H(x_0,r)= 2 r^{1-N} \int_{B_r(x_0)} |\nabla u|^2 +|\nabla v|^2 + 2u^2 v^2 \ge 0:
\]
for every $x_0 \in \R^N$ and $r>0$ the function $H(x_0,r)$ is nondecreasing in $r$.
\end{remark}
Proposition 5.2 of \cite{BeTeWaWe} says that also the Almgren quotient is nondecreasing as function of $r$. 

\begin{proposition}[Almgren monotonicity formula]\label{Almgren formula}  
Let $(u,v)$ be a solution of \eqref{system}, let $x_0 \in \R^N$. The Almgren frequency function $N(x_0,r)$ is well defined for $r\in (0,+\infty)$, nonnegative and nondecreasing in $r$.
\end{proposition}
A control on the Almgren frequency function gives useful information about the growth of the function $H$ with respect to the radial variable. The proof of the following result is a straightforward modification of the proof of Proposition 5.3 in \cite{BeTeWaWe}
\begin{corollary}\label{doubling}
Let $(u,v)$ be a solution of \eqref{system}, let $x_0 \in \R^N$, and assume that $d_1 \le N(x_0,r) \le d_2$ for $0<R_1<r<R_2$. Then
\[
\frac{r_2^{2d_1}}{r_1^{2d_1}} \le \frac{H(x_0,r_2)}{H(x_0,r_1)} \le e^{d_2} \frac{r_2^{2d_2}}{r_1^{2d_2}}
\]
for every $R_1<r_1<r_2<R_2$. 
\end{corollary}
In light of the subharmonicity of $(u,v)$, it is not difficult to deduce a pointwise estimate on the growth of the solution $(u,v)$.
\begin{corollary}\label{from N to growth}
Let $(u,v)$ be a solution of \eqref{system}, let $x_0 \in \R^N$ and $p \ge 1$, and assume that $N(x_0,r) \le p$ for every $r>0$. Then there exists $C>0$ such that
\[
u(x)+v(x) \le C(1+|x|^p) \qquad \forall x \in \R^N.
\]
\end{corollary}
\begin{proof}
The thesis follows if we show that there exists $C>0$ such that
\[
u(x)+v(x) \le C(1+|x-x_0|^p) \qquad \forall x \in \R^N.
\]
Suppose by contradiction that our claim is not true. Then there exists $r_n \to +\infty$ such that
\beq\label{growth more then d}
\lim_{n \to +\infty} \frac{u(x_0+ r_n x)}{r_n^p}=+\infty
\eeq
for some $x \in \mathbb{S}^{N-1}$ and $r_n \to +\infty$. In light of Corollary \ref{doubling}, we have
\begin{equation}\label{eq40}
\frac{H(x_0,2r_n)}{(2r_n)^{2p}} \le e^p H(x_0,1) \quad \Rightarrow  \quad \int_{\pa B_{2r}(x_0)} u^2+v^2 \le C r_n^{2p+N-1}.
\end{equation}
As $u$ is subharmonic, $u \le \f_n$ in $B_{2 r_n}(x_0)$, where $\f_n$ is the solution of
\[
\begin{cases}
-\D \f_n=0 & \text{in $B_{2r_n}(x_0)$}\\
\f_n=u & \text{on $\pa B_{2r_n}(x_0)$}.
\end{cases}
\]
By the representation formula for harmonic functions we know that for every $x \in \overline{B}_{r_n}(x_0)$
\begin{align*}
\f_n(x) & = \frac{4 r_n^2 - |x-x_0|^2}{2N|\mathbb{S}^{N-1}| r_n} \int_{\pa B_{2 r_n}(x_0) } \frac{u(x)}{|x-y|^N} \,d\sigma_y \\
& \le C r_n \left( \int_{\pa B_{2 r_n}(x_0) } \frac{d\sigma_y} {r_n^{2N}}\right)^{\frac{1}{2}}\left( \int_{\pa B_{2 r_n}(x_0) } u^2\right)^{\frac{1}{2}} \le C r_n^{-\frac{N-1}{2}+p + \frac{N-1}{2}}= C r_n^p,
\end{align*} 
where $C$ depends only on the dimension $N$, and for the last inequality we used the \eqref{eq40}. Thus, for every $x \in \S^{N-1}$ we obtain
\[
u(x_0+r_n x) \le \f_n(x) \le C r_n^p \qquad \forall n,
\]
in contradiction with equation \eqref{growth more then d}.
\end{proof}
As proved in \cite{Fa2}, the converse holds true.
\begin{lemma}[Lemma 2.1 in \cite{Fa2}]\label{from growth to N}
Let $(u,v)$ be a solution of \eqref{system}, let $x_0 \in \R^N$, and assume that there exist $p \ge 1$ and $C>0$ such that
\[
u(x)+v(x) \le C(1+|x|^p) \qquad \forall x \in \R^N.
\]
Then $N(x_0,r) \le p$ for every $x_0 \in \R^N$ and for every $r>0$.
\end{lemma}

\begin{remark}
Combining Corollary \ref{from N to growth} and Lemma \ref{from growth to N}, we deduce that if for a single $x_0 \in \R^N$ we know that $N(x_0,r) \le p$ for every $r >0$, then 
\[
u(x)+v(x) \le C(1+|x|^p) \qquad \forall x \in \R^N,
\]
so that $N(x,r) \le p$ for every $x \in \R^N$. That is, a bound of the Almgren quotient centered in a point $x_0 \in \R^N$ provides the same bound for the quotients $N(x,\cdot)$ for every $x \in \R^N$.
\end{remark}

\begin{remark}\label{rem su scaling per N}
We point out that all these results hold true for a solution $(u_\beta,v_\beta)$ of \eqref{system beta}, with $E(x_0,r)$ replaced by the corresponding energy function, that is,
\[
\frac{1}{r^{N-2}} \int_{B_r(x_0)} |\nabla u_\beta|^2 + |\nabla v_\beta|^2 + \beta u_\beta^2 v_\beta^2.
\] 
 \end{remark}

\subsection*{The blow-down family}

By means of the previous monotonicity formulae, in \cite{BeTeWaWe} it is proved that the asymptotic information about $\{(u_\beta,v_\beta)\}$ can be improved for particular sequences. Let $(u,v)$ be a solution of \eqref{system}. For every $x_0 \in \R^N$ and $R>0$, recall that we introduced the blow-down family
\[
\left(u_{x_0,R}(x), v_{x_0,R}(x) \right) := \left( \frac{1}{\sqrt{H(x_0,R)}} u(x_0+Rx), \frac{1}{\sqrt{H(x_0,R)}} v(x_0 +Rx) \right).
\]
By definition, $\int_{\pa B_1(0)} u_{x_0,R}^2 + v_{x_0,R}^2=1$ for every $x_0 \in \R^N$ and $R>0$. Also, $(u_{x_0,R},v_{x_0,R})$ solves 
\begin{equation}\label{system for blow-down}
\begin{cases}
-\Delta u_{x_0,R}= - H(x_0,R) R^2 \, u_{x_0,R} \, v_{x_0,R}^2 & \text{in $\R^N$} \\
-\Delta v_{x_0,R}=- H(x_0,R) R^2 \,  u_{x_0,R}^2 \, v_{x_0,R} & \text{in $\R^N$} \\
u_{x_0,R},v_{x_0,R}>0 & \text{in $\R^N$}.
\end{cases}
\end{equation}

\begin{remark}\label{N ed N riscalato}
A direct computation shows that if $N(x_0,r) \le p$ for every $r \ge 1$, the same estimate holds true for the Almgren quotient associated to the function $(u_{x_0,R},v_{x_0,R})$ (for every $x_0 \in \R^N$ and $R>0$):
\[
\frac{\frac{1}{r^{N-2}} \int_{B_r(0)} |\nabla u_{x_0,R}|^2 +|\nabla v_{x_0,R}|^2 + H(x_0,R) R^2 \, u_{x_0,R}^2 \, v_{x_0,R}^2  }{\frac{1}{r^{N-1}} \int_{\pa B_r(0)} u_{x_0,R}^2 + v_{x_0,R}^2} = N(x_0, R r ) \le p \qquad \forall r \ge 1.
\]
As a consequence, if we can bound $N(x_0,\cdot)$, we can apply Corollary \ref{doubling} on $(u_{x_0,R}, v_{x_0,R})$.
\end{remark}

Theorem 1.4 in \cite{BeTeWaWe} says, roughly speaking, that if the Almgren frequency function is bounded, then the limit of $N(x_0,r)$ as $r \to +\infty$ (which exists by monotonicity) is a positive integer and the limiting profile is a homogeneous harmonic polynomial. It is straightforward to check that, although therein it is considered the case $x_0=0$, the result holds true for any $x_0 \in \R^N$. 

\begin{theorem}\label{thm blow down}
Let $(u,v)$ be a solution of \eqref{system}, let $x_0 \in \R^N$, and assume that 
\[
\lim_{r \to +\infty} N(x_0,r) =: d_{x_0} < +\infty.
\] 
Then $d_{x_0}$ is a positive integer. There exist a subsequence of the blow down family $\{(u_{x_0,R},v_{x_0,R}): R >0\}$, denoted $\{(u_{x_0,R_n},v_{x_0,R_n})\}$, and a homogeneous harmonic polynomial of degree $d_{x_0}$, denoted by $\Psi_{x_0}$, such that $(u_{x_0,R_n},v_{x_0,R_n}) \to (\Psi_{x_0}^+,\Psi_{x_0}^-)$ as $R \to +\infty$ in $\mathcal{C}^0_{loc}(\R^N)$ and in $H^1_{loc}(\R^N)$. Moreover,
\[
H(x_0,R) R^2 \, u_{x_0,R_n}^2 \, v_{x_0,R_n}^2 \to 0 \qquad \text{in $L^1_{loc}(\R^N)$}.
\]
\end{theorem} 

This achievement permits to say something more on the asymptotic of $H(x_0,\cdot)$ in case $(u,v)$ has algebraic growth. 

\begin{corollary}\label{cor blow-down}
Let $(u,v)$ be a solution of \eqref{system} with algebraic growth. For $x_0 \in \R^N$, let \( d_{x_0}= \lim_{r \to +\infty} N(x_0,r)\), which is a positive integer by the previous statement. For every $\eps>0$ it results
\[
\lim_{r \to +\infty} \frac{H(x_0,r)}{r^{2d_{x_0}(1-\eps)}}=+\infty.
\]
\end{corollary}
\begin{proof}
As $d_{x_0} \ge 1$, using the Almgren monotonicity formula (Theorem \ref{Almgren formula}) we deduce that for every $\eps>0$ there exists $r_\eps>0$ such that if $r>r_\eps$ then
\[
N(x_0,r) \ge d_{x_0}\left(1-\frac{\eps}{2}\right).
\]
Hence, we can use Corollary \ref{doubling} to obtain
\[
H(x_0,r)\ge C r^{2d_{x_0} \left(1-\frac{\eps}{2}\right)} \qquad \forall r > r_\eps,
\]  
with $C>0$. Therefore
\[
\lim_{r \to +\infty} \frac{H(x_0,r)}{r^{2d_{x_0}(1-\eps)}} \ge \lim_{r \to +\infty} C\frac{r^{2d_{x_0} \left(1-\frac{\eps}{2}\right)}}{r^{2d_{x_0}(1-\eps)}} = +\infty. \qedhere
\]
\end{proof}

\subsection*{An Alt-Caffarelli-Friedman monotonicity formula}

For a solution $(u,v)$ to \eqref{system}, recall the definition
\[
J(x_0,r)= \frac{1}{r^4} \int_{B_r(x_0)} \frac{|\nabla u(y)|^2 + u^2(y) v^2(y)}{|y-x_0|^{N-2}}\,dy  \int_{B_r(x_0)} \frac{|\nabla v(y)|^2 + u^2(y) v^2(y)}{|y-x_0|^{N-2}}\,dy.
\]
First of all, we report the useful formula (4.11) in \cite{Wa}: there exists $C>0$ independent on $x_0 \in \R^N$ and on $r \ge 1$ such that
\begin{equation}\label{4.11 by Wang}
\frac{1}{r^2} \int_{B_r(x_0)} \frac{|\nabla u(y)|^2 + u^2(y) v^2(y)}{|y-x_0|^{N-2}}\,dy \le \frac{C}{r^{N+2}}\int_{B_{2r}(x_0)} u^2.
\end{equation}
Recently, K. Wang proved an Alt-Caffarelli-Friedman monotonicity formula which enhances a previous similar result in \cite{NoTaTeVe}. 

\begin{theorem}[Theorem 4.3 in \cite{Wa}]\label{ACF}
Let $(u,v)$ be a solution of \eqref{system} satisfying \eqref{alg growth}, let $x_0 \in \R^N$. There exists $C(x_0)>0$ such that
\[
r \mapsto e^{-C(x_0) r^{-1/2}} J(x_0,r) \quad \text{is nondecreasing in $r$}
\]
for every $r \ge 1$.
\end{theorem}

\noindent \textbf{Acknowledgments:} the first author is supported by the ERC grant EPSILON ({\it Elliptic Pde's and Symmetry of Interfaces and Layers for Odd Nonlinearities}). The second author thanks Prof. Susanna Terracini for many inspiring discussions related to this problem, and Kelei Wang for some useful comments concerning his preprint {\it On the De Giorgi type conjecture for an elliptic system modeling phase separation}. The second author is partially supported by PRIN 2009 grant {\it Critical Point Theory and Perturbative Methods for Nonlinear Differential Equations}.

 \end{document}